\documentclass[11pt]{amsart}
\usepackage{amssymb, latexsym, amsmath, amsfonts, amsrefs}
\usepackage{hyperref, enumitem, fancyhdr}

\newtheorem{thm}{Theorem}[section]
\newtheorem{cor}[thm]{Corollary}
\newtheorem{lem}[thm]{Lemma}
\newtheorem{prop}[thm]{Proposition}
\theoremstyle{definition}

\theoremstyle{remark}

\numberwithin{equation}{section}
\theoremstyle{remark}
\newtheorem{exam}[thm]{Example}

\newcommand{\mbb}{\mathbb}
\newcommand{\ra}{\rightarrow}
\newcommand{\z}{\zeta}
\newcommand{\pa}{\partial}
\newcommand{\ov}{\overline}

\newcommand{\ep}{\epsilon}
\newcommand{\no}{\noindent}
\newcommand{\al}{\alpha}
\newcommand{\Om}{\Omega}
\newcommand{\cal}{\mathcal}
\newcommand{\ti}{\tilde}
\newcommand{\la}{\lambda}

\newcommand{\abs}[1]{\left\vert#1\right\vert}
\newcommand{\de}{\delta}

\newcommand{\La}{\Lambda}

\newcommand{\ga}{\gamma}
\newcommand{\Ga}{\Gamma}
\newcommand{\be}{\beta}

\begin{document}
\title{Comments on the Green's function of a planar domain}
\thanks{The first named author was supported by the DST-INSPIRE grant IFA-13 MA-21.
The last named author was supported by the DST Swarnajayanti
Fellowship 2009--2010 and a UGC--CAS Grant}
\author{Diganta Borah, Pranav Haridas and Kaushal Verma}

\address{Diganta Borah: Indian Institute of Science Education and Research, Pune, India}
\email{dborah@iiserpune.ac.in}

\address{Pranav Haridas: Department of Mathematics, Indian Institute of Science, Bangalore 560 012, India}
\email{pranav10@math.iisc.ernet.in}

\address{Kaushal Verma: Department of Mathematics, Indian Institute of Science, Bangalore 560 012, India}
\email{kverma@math.iisc.ernet.in}

\begin{abstract}
We study several quantities associated to the Green's function of a multiply connected domain in the complex plane. Among them are some intrinsic properties such as geodesics, curvature, and $L^2$-cohomology of the capacity metric and critical points of the Green's function. The principal idea used is an affine scaling of the domain that furnishes quantitative boundary behaviour of the Green's function and related objects.
\end{abstract}

\maketitle

\section{Introduction}

\no Let $D \subset \mbb C$ be a regular domain. Recall that the Green's function
$G_D(z,p)$ of $D$ with pole at $p \in D$ is defined by
\[
G_D(z,p) = -\log \vert z - p \vert + H_D(z,p)
\]
where $H_D(z,p)$ is the unique harmonic function of $z$ in $D$ with boundary values
$\log \vert z - p \vert$. The existence of $H_D(z,p)$ is guaranteed by the
solvability of the Dirichlet problem on regular domains and the uniqueness is a
consequence of the maximum principle for harmonic functions. Thus $G_D(z,p)$ is
the unique function satisfying the following properties: it is harmonic on $D \setminus \{p\}$,
$G_D(z,p) \to 0$ as $z \to \pa D$ and $G_D(z,p)+\log \vert z - p \vert$ is harmonic near
$p$. It is well known that $G_D(z,p)$ is symmetric in $z, p$
and hence $H_D(z,p)$ inherits the same property. Therefore, $H_D(z,p)$ is harmonic in
both $z, p \in D$. Being separately harmonic implies that $H_D(z,p)$ is harmonic
on $D \times D$ -- see for instance \cite{L}, \cite{H}. The function $H_D(z,p)$ is therefore
the regular part of $G(z,p)$.

\medskip

The purpose of this paper is to study several intrinsic quantities, all of whom
owe their existence to $G_D(z, p)$. The reader is referred to the recent work of Gustafsson--Sebbar \cite{GS} that also
touches upon some of these themes among many others. To elaborate on the
first of these, note that in a neighborhood of a given $p \in D$, $H_D(z, p)$ is the real part of a holomorphic
function $h_D(z, p)$ (of $z$) which is uniquely determined by choosing the imaginary part $\Im h_D(p, p) = 0$. Let
\[
h_D(z, p) = c_0(p) + c_1(p)(z-p) + \ldots + c_n(p)(z-p)^2 + \ldots,
\]
near $p$. Note that
\[
\La_D(p)=\lim_{z \to p} \big( G_D(z,p) + \log\vert z - p \vert \big)
\]
exists and this is the \textit{Robin constant} for $D$ at $p$. In other words, $\La_D(p) = H_D(p,p) = c_0(p)$.
This implies that $\La_D(p)$, and hence $c_0(p)$, are real analytic on $D$. The correspondence $p \mapsto \La_D(p)$ is the
Robin function for $D$. The constant $c_D(p) = e^{-\La_D(p)}$ is the \textit{capacity constant} for $D$ at $p$ and the
correspondence $p \mapsto c_D(p)$ will be referred to as the capacity function for $D$. Under a conformal map $\phi : D \to \Om$, the
invariance of the Green's function implies that
\begin{equation}\label{trf}
\begin{split}
\La_{D}(p) & =\La_{\Om}\big(\phi(p)\big)-\log\big\vert \phi^{\prime}(p)\big\vert, \;\text{and}\\
c_{D}(p) & =c_{\Om}\big(\phi(p)\big) \big\vert \phi^{\prime}(p) \big\vert.
\end{split}
\end{equation}
The first of these can be regarded as a transformation rule for $c_0(p)$ since $\La_D(p) = c_0(p)$ while the second one shows that $c_D(z) \vert dz \vert$ is a conformal metric
-- the capacity metric on $D$. Various differential geometric aspects of this metric depend on understanding the first few coefficients $c_i(p)$ in the expansion of $h_D(z, p)$.
As examples, the curvature depends on the second derivative of $c_0(p)$ while the associated Levi-Civita connection is given by the derivative of $c_1(p)$. Furthermore, it was observed
in \cite{GS} that a suitable combination of $c_1(p)$ and $c_2(p)$ transforms as a projective connection.

\medskip

\begin{exam}
For the unit disc $\mathbb D$ and $p \in \mathbb D$,
\[
G_{\mathbb D}(z,p)= -\log \vert z - p \vert + \log \vert 1 - \ov p z \vert,
\]
and therefore
\[
\La_{\mathbb D}(p)= \log\big(1 - \vert p \vert^2\big).
\]
Hence
\[
c_{\mathbb D}(p)= \big(1-\vert p \vert^2\big)^{-1},
\]
which means that the capacity metric coincides with the hyperbolic metric. Also, $h_{\mathbb{D}}(z,p)=\log(1-\ov p z)$, and so for $n \geq 1$, the coefficients
\[
c_{n,\mbb{D}}(p)=\frac{1}{n!} \frac{\pa^n}{\pa z^n} \log (1-\ov p z)\Big\vert_{z=p}=-\frac{\ov p^n}{n\big(1-\vert p \vert^2\big)^n}.
\]
\end{exam}

\begin{exam}\label{ex-half}
Consider the half plane
\[
\cal H=\left\{z \in \mbb{C} : 2 \Re (az) + k <0 \right\},
\]
where $a$ is a nonzero complex number and $k$ is a real constant. Then
\[
G_{\cal H}(z,p)=-\log \vert z - p \vert + \log \vert z - p^{*} \vert,
\]
where
\[
p^{*}=p-\frac{2\Re(ap)+k}{a},
\]
is the symmetric point of $p$ with respect to the boundary $\pa \cal H$. Thus
\begin{equation}\label{La-hplane}
\La_{\cal H}(p)=\log \vert p - p^{*} \vert=\log \big\vert 2\Re(ap)+k \big\vert -\log \vert a \vert,
\end{equation}
and
\begin{equation}\label{cap-hplane}
c_{\cal H}(p)=\frac{\vert a \vert}{\big\vert 2 \Re(ap)+k \big\vert}.
\end{equation}
Also, $h_{\mathcal{H}}(z,p)=\log(z-p^{*})-i\arg(p-p^{*})$, where $\arg$ is the principal argument and so for $n\geq1$, the coefficients
\begin{equation}\label{c_n-hplane}
c_{n,\cal H}(p)=\frac{1}{n!}\frac{\pa^n}{\pa z^n} \log(z-p^{*}) \Big\vert_{z=p}=\frac{(-1)^{n-1} a^n}{n \big(2\Re(a)+k\big)^n}.
\end{equation}
\end{exam}
In both examples it can be seen that the capacity function and the coefficients blow up near the boundaries at a rate
which is of the order of some power of the reciprocal of the distance to the boundary. The following theorem shows that this holds in general and
at the same time generalizes an observation regarding this made in \cite{GS}--see Lemma 5.3 therein.

\begin{thm}\label{c_D,c_n-asymp}
Let $D \subset \mbb C$ be a regular domain with a $C^2$-smooth
open piece $\Ga \subset \pa D$. Let $p_0 \in \Ga$ and let $\psi$ be a
$C^2$-smooth local defining function for $D$ near $p_0$, i.e., $U \cap D = \{ \psi < 0\}$ for some neighborhood $U$ of $p_0$ and $d\psi \not= 0$ on $\Gamma$. Then, as $p \to p_0$:
\begin{enumerate}
\item[(i)] $\La_D(p) - \log\big(-\psi(p)\big) \to -\log \big\vert \pa \psi(p_0) \big\vert$. Furthermore, for all non-negative integers $\al, \be$ such that $(\al, \be) \not= (0, 0)$
\[
\pa^{\al+ \be}\La(p) \big(-\psi(p)\big)^{\al+\be} \to
-(\al+\be-1)!\big(\pa \psi(p_0)\big)^{\al} \big(\ov \pa \psi(p_0) \big)^{\be}.
\]

\item[(ii)] For all $n \geq 1$ and nonnegative integers $\al, \be$
\[
\pa^{\al+\be}c_{n, D}(p)\big(-\psi(p)\big)^{n+\al+\be} \to -\frac{(n+\al+\be-1)!}{n!}\big(\pa
\psi(p_0)\big)^{n+\al}\big(\ov \pa \psi(p_0) \big)^{\be}.
\]
\end{enumerate}
\end{thm}
Here and henceforth, we will follow the standard convention of denoting complex partial derivatives by powers of $\pa$ and $\ov \pa$:
\[
\pa^\al=\frac{\pa}{\pa z^{\al}}, \quad \ov \pa^{\be}=\frac{\pa}{\pa \ov z^{\be}}, \quad \pa^{\al+\be}=\pa^{\al}\ov\pa^{\be}=\frac{\pa^{\al+\be}}{\pa z^{\al}\pa \ov z^{\be}}.
\]
Since $c_D(p) = e^{-\La_D(p)}$, the statements in (i) above can be translated to give the boundary asymptotics of $c_D$ and {\it all} of its derivatives. As an example, it follows that
\begin{equation}
c_D(p)\big(-\psi(p)\big) \to \big\vert \pa \psi(p_0) \big\vert
\end{equation}
as $p \to p_0$. Consequently, $c_D(z)$ blows up at the rate of $\big(-\psi(z)\big)^{-1}$ (which is the same as the reciprocal of the distance of $z$ to $\pa D$ by the smoothness of $\psi$) near the boundary.  In the neighborhood $U$ of $p_0$, the hyperbolic metric on $U \cap D$ has the same rate of blow up near $U \cap \pa D$. Thus these metrics are asymptotically the same in $U \cap D$ and this naturally leads to a comparison of their various geometric aspects. To start with, recall that the capacity metric is in general distance decreasing under holomorphic mappings, a property enjoyed by the hyperbolic metric as well. A theorem of Minda \cite{Mi1} shows that on a hyperbolic Riemann surface, the capacity metric is dominated by the hyperbolic metric and that equality at a single point forces the Riemann surface to be simply connected. Since $D$ is assumed to be an arbitrary regular domain, these metrics do not coincide anywhere on it. A qualitative description of the geodesics for the capacity metric on the standard annulus $A = \big\{r < \vert z \vert < 1\big\}$ for $0 < r < 1$ is also available in \cite{Ab} (which relies on a more general result of Herbort \cite{H} that applies to conformal metrics of a specific form on $A$) and finally we note Blocki's affirmative solution \cite{Bl} of Suita's conjecture (which asked whether $c_D^2(z) \le \pi K_D(z)$ where $K_D(z)$ is the Bergman kernel on the diagonal; see \cite{Su}). We will focus on the boundary behavior of the curvature and geodesics among other invariants attached to this conformal metric. Recall that the curvature of $c_D(z) \vert dz \vert$ is given by
\[
\mathcal K(z) = -4c_D^{-2}(z) \;\pa \ov \pa \log c_D(z)
\]
and by \cite{Bl}, \cite{Su} it follows that $\mathcal K \leq -4$ everywhere on a bounded domain $D$ and furthermore, if $\pa D$ is sufficiently smooth then $\mathcal K(z) \to -4$ as $z$ approaches the boundary $\pa D$ -- see \cite{Su} for further pertinent remarks that formed the genesis of this conjecture. That the boundary behavior of $\mathcal K$ is also a consequence of Theorem 1.3 is shown in the following, which incidentally emphasizes the local nature of this phenomenon:

\begin{prop}
Let $D$, $\Ga$ and $p_0$ be as in Theorem 1.3. Then $\mathcal K(p) \to -4$ as $p \to p_0$.
\end{prop}

Next, we study the global behavior of geodesics in the capacity metric and in what follows, $D$ will be assumed to have connectivity at least $2$ and $C^{\infty}$-smooth boundary everywhere. Any reduction in the smoothness of the boundary does not lead to any further generality since it is known that a planar domain whose boundary consists of finitely many continua is conformally equivalent to one that has $C^{\infty}$-smooth boundary. Let $\psi$ be a $C^{\infty}$-smooth global defining function for $\pa D$. Note that the capacity metric is complete on such a domain $D$ since it is uniformly comparable with the hyperbolic metric near $\pa D$. A direct consequence of this is that every nontrivial homotopy class of loops in $D$ contains a closed geodesic in the capacity metric -- this follows from Theorem 1.1 in \cite{He}. On the other hand, non-closed geodesics can either diverge to the boundary $\pa D$ as $t \to \pm \infty$ or exhibit spiral--like behavior. To make this precise in this case, a smooth path $z : \mathbb R \to D$ is a geodesic spiral if it is a non-closed geodesic for the capacity metric that lies in a compact subset $K \subset D$ for all time $t \geq 0$. 

\medskip

The differential equation for geodesics in the capacity metric $ds = c_D(z) \vert dz \vert$ takes the form
\[
z''(t) = \pa \La_D\big(z(t)\big)\big(z'(t)\big)^2.
\]
Suppose that $z(t)$ diverges to the boundary as $t \to +\infty$. To analyse this case, a mixture of two inputs are used -- one, calculations similar to those of Fefferman (\cite{Fe}) for the Bergman metric on strongly pseudoconvex domains show that $z(t)$ approaches the boundary $\pa D$ at an exponential decaying rate and two, we interpret Theorem 1.3 (i) as saying that the capacity metric on a smoothly bounded domain is Gromov hyperbolic since it is comparable with the hyperbolic metric. Thus $z(t)$ can be thought of as a quasigeodesic for the hyperbolic metric on $D$. By using well known estimates for the shape of the hyperbolic balls on $D$, it is possible to show that $z(t)$ converges to a unique point on $\pa D$. The other case when $z(t)$ spirals can be dealt with by using the estimates from Theorem 1.3 and some arguments from \cite{He}. All this can be summarized as follows:

\begin{thm}
Let $D \subset \mbb C$ be a non simply connected, smoothly bounded domain equipped with the capacity metric. Then
\begin{enumerate}
\item[(i)] Every nontrivial homotopy class of loops in $D$ contains a closed geodesic.

\item[(ii)] Every geodesic $z(t)$ that does not stay in a compact set of $D$ for all time $t \geq 0$ hits the boundary $\pa D$ at a unique point.

\item[(iii)] For every $z_0 \in D$ that does not lie on a closed geodesic, there exists a geodesic spiral passing through $z_0$.
\end{enumerate}
\end{thm}

Another consequence of Theorem 1.3 is the following observation  about the Euclidean curvature of the geodesics in the capacity metric which is similar to a result of Minda \cite{Mi2} (this was also noted in \cite{GS}) who worked with the hyperbolic metric on convex domains in $\mbb C$.

\begin{cor}
Let $D$ be a smoothly bounded domain in $\mbb C$. Suppose $z(s)$ is a geodesic
of the capacity metric which is parametrised by Euclidean arc length $s$. Then its Euclidean curvature
$\kappa\big(z(s)\big)$ satisfies
\[
\bigg \vert\kappa\big(z(s)\big)\Big(-\psi\big(z(s)\big)\Big) \bigg\vert \approx 1 
\]
for all $s$.
\end{cor}

In other words, $\kappa(z)$ essentially behaves as the reciprocal of the distance of $z$ to $\pa D$. Theorem 1.3 is also useful in computing the $L^2$-cohomology of $D$ (smoothly bounded as always) endowed with the capacity metric $ds = c_D(z) \vert dz \vert$. Let $\Om^{k}_{2}$ be the space of $k$-forms on $D$ which are square integrable with
respect to $ds^2$. Then the $L^2$-cohomology of the complex
\[
\Om^0_2 \xrightarrow{d_0} \Om^1_2 \xrightarrow{d_1} \Om^2_2 \xrightarrow{d_2}0
\]
is defined by
\[
H^{k}_{2}(D) = \ker d_k / \ov{\text{im}\,d_{k-1}}
\]
where the closure is taken in the $L^2$ norm. Since $ds$ is complete,
these cohomology groups are completely determined by the space $\mathcal{H}^k_2(D)$
of square integrable harmonic forms:
\[
H^{k}_{2}(D) \cong \mathcal{H}^k_2(D).
\]
We also have the decomposition
\[
\mathcal{H}^k_2(D) = \oplus_{p+q=k} \mathcal{H}^{p,q}_2(D).
\]
\begin{thm}
Let $D \subset \mathbb C$ be a smoothly bounded domain. Let $\mathcal{H}^{p,q}_2(D)$ be
the space of square integrable harmonic $(p,q)$-forms on $D$ relative to $ds$.
Then
\[
\dim \mathcal{H}^{p,q}_2(D) =
\begin{cases}
0, & \text{if $p+q \neq 1$}\\
\infty, & \text{if $p+q=1$}.
\end{cases}
\]
\end{thm}

Results of this kind for the Bergman metric on strongly pseudoconvex domains in $\mbb C^n$
were obtained by Donnelly--Fefferman and Donnelly (see \cite{DF} and \cite{D}) and in a more
general setup by McNeal \cite{Mc} and Ohsawa \cite{Oh} among others. The final result relates
the critical points of the Green's functions of a family of variable domains that converge to a limiting
domain, and the zeros of the Bergman kernel of the limiting domain. This extends a result of Solynin \cite{So},
and Gustafsson--Sebbar \cite{GS}.

\begin{thm}
Let $D \subset \mathbb C$ be a smoothly bounded domain and $D_k \subset \mathbb C$ a sequence of smoothly bounded domains
that converge to $D$ in the $C^{\infty}$-topology. Then, for $(z_0, \z_0) \in D \times \pa D$, the Bergman kernel $K_D(z_0, \z_0) = 0$
if and only if there exists a subsequence $(z_{k_m}, \z_{k_m}) \in D_{k_m} \times D_{k_m}$ converging to $(z_0, \z_0)$ such that
\[
\pa G_{k_m}(z_{k_m}, \z_{k_m}) = 0,
\]
where $\pa=\pa/\pa z$ is the derivative with respect to the first variable.
\end{thm}

\section{Boundary behavior of $\La_D$}

\subsection{Proof of Theorem 1.3}
\noindent Since $c_D(p) = e^{-\La_D(p)}$ and $\La_D(p)$ depends on $G_D(z, p)$, it suffices to understand the variation of $G_D(z, p)$ with respect to $p$. This
will be done in three steps. The first of these records several useful observations about $\La_D(p)$ and its relation with $G_D(z, p)$. While several other properties of $\La_D$, including the
fact that it is a superharmonic function of $p$, can be found in \cite{Ya}, it will be sufficient for us to only mention the relevant ones listed in Proposition 2.1 below.
The second step involves a rescaling of $D$ near $p_0$ by affine maps. This produces a sequence of domains that converge to a half space in an appropriate sense. In doing so, the question of the boundary behavior of $\La_D$ reduces to an interior problem about the convergence of the Green's functions of these domains. Appealing to Step $1$ then yields information about all derivatives of $\La_D$ near $p_0$. In the final step, this is translated in terms of $c_D(p)$ from which the desired boundary asymptotics can be read off.

\begin{prop}
Let $D \subset \mathbb C$ be a regular domain. For every disc $B(q_0,r)$ which is compactly contained in $D$,
\begin{equation}\label{repn-La-1}
\La_D(q_0)=\log r + \frac{1}{2\pi}\int_{-\pi}^{\pi}G_D(q_0+re^{i\theta}, q_0) \,d\theta
\end{equation}
and for every $q  \in B(q_0, r)$,
\begin{equation}\label{repn-La-2}
\La_D(q)=\frac{1}{4\pi^2} \int_{[-\pi,\pi]^2}
H_D(q_0 + re^{i\theta},q_0 + re^{i\phi}) \frac{\big(r^2- \vert q - q_0\vert^2\big)^2}
{\vert q_0+re^{i\theta} - q \vert^2 \vert q_0+re^{i\phi} - q \vert^2}\, d\theta\,
d\phi.
\end{equation}
Finally, if $D$ is bounded regular, then $\La_D(p) \to -\infty$ as $p \to \pa D$.
\end{prop}

\begin{proof}
Integrating
\[
G_D(z,q_0)=-\log \vert z - q_0 \vert + H_D(z,q_0)
\]
over $\pa B(q_0, r)$, we obtain
\[
\frac{1}{2\pi}\int_{-\pi}^{\pi} G_D(q_0+r e^{i\theta}, q_0) \, d\theta =
-\log r + \frac{1}{2\pi} \int_{-\pi}^{\pi} H_D(q_0+re^{i\theta}, q_0) \, d\theta.
\]
The last term is the mean value of the harmonic function $H_D(z,q_0)$ on $\pa B(q_0, r)$
and hence is equal to $H_D(q_0,q_0)=\La_D(q_0)$.

\medskip

By a repeated application of the Poisson integral formula,
\[
\La_D(q)=\frac{1}{4\pi^2}\int_{-\pi}^{\pi} \int_{-\pi}^{\pi}
H(q_0+re^{i\theta}, q_0+re^{i\phi}) \frac{r^2- \vert q - q_0\vert^2}{\vert
q_0+re^{i\theta} - q \vert^2}\frac{r^2 - \vert q-q_0\vert^2 }{\vert
q_0+re^{i\phi} - q \vert^2} \, d\theta \, d\phi,
\]
for all $q \in B(q_0, r)$. The continuity of $H_D$ on $D \times D$ along with
Fubini's theorem gives (2.2).

\medskip

For the final assertion, let $p_0 \in \pa D$ and $M>0$ be given. Choose $r>0$ such that $\log \vert z - p
\vert < -M$ for all $z,p$ in the disc $B = B(p_0,r)$. Let $u$ be the harmonic function on $D$ with boundary values $-M$ on
$\pa D \cap B$ and $\log(2 \vert z - p_0 \vert)$ on $\pa D \setminus B$. For $p \in B \cap D $, let
\[
s_p(z)=u(z) - H(z,p), \quad z \in D.
\]
Then $s_p(z)$ is a harmonic function on $D$ with boundary values
\[
s_p(z) = -M - \log \vert z - p \vert > 0,
\]
on $B \cap \pa D $ and
\[
s_p(z)=\log \big(2 \vert z - p_0\vert\big) - \log \vert z-p \vert >0
\]
on $\pa D \setminus B$ as
\[
\vert z - p \vert \leq \vert z - p_0 \vert + \vert p-p_0 \vert \leq \vert
z-p_0\vert + r \leq 2\vert z - p_0 \vert.
\]
By the maximum principle $s_p(z) \geq 0$ on $D$ and in particular $s_p(p) \geq
0$. Hence $\La_D(p) \leq u(p)$ for all $p \in B \cap D $. Consequently
\[
\limsup_{p \to p_0} \La(p) \leq \limsup_{p \to p_0} u(p) = -M
\]
which implies that $\La_D(p) \to -\infty$ as $p \to p_0$.

\end{proof}

\medskip

\noindent Let $D \subset \mathbb C$ be as in Theorem 1.3, i.e., there is a
$C^2$-smooth open piece $\Ga \subset \pa D$ for which there is a $C^2$-smooth local
defining function $\psi$ near $p_0$. Let $\{p_j\}$ be a
sequence in $D$ converging to $p_0$ and without loss of generality we may assume
that $\psi(p_j)$ is defined for all $j\geq 1$. Consider the affine maps
\[
T_j(z)=\frac{z-p_j}{-\psi(p_j)}
\]
and let $D_j=T_j(D)$. Observe that $\psi \circ T_j^{-1}$ is a local defining
function for $D_j$ at $T_j(p_0)$ and
\[
\psi \circ T_j^{-1}(z)=\psi\Big(p_j+z\big((-\psi(p_j)\big)\Big)
=\psi(p_j) + 2 \Re \big(\pa \psi (p_j) z\big)\big(-\psi(p_j)\big) + \psi^2(p_j)\;O(1).
\]
Therefore,
\[
\psi_j(z)=\frac{\psi \circ T_j^{-1}(z)}{-\psi(p_j)} = -1 + 2 \Re\big(\pa
\psi(p_j) z\big) +  \big(-\psi(p_j)\big) O(1)
\]
is again a local defining function for $D_j$ at $T_j(p_0)$ and in the limit it
can be seen that these functions converge to
\[
\psi_{\infty}(z)=-1+2\Re\big(\pa \psi(p_0) z\big)
\]
in the $C^2$-topology on every compact subset of $\mathbb C$. In particular, this
implies that the domains $D_j$ converge to the half plane
\begin{equation}\label{half-p}
\mathcal{H}=\Big\{z \in \mathbb C : 2\Re\big(\pa \psi(p_0) z\big) - 1 <0\Big\}
\end{equation}
in the Hausdorff sense. Let $G_j$ be the Green function for $D_j$ and $\La_j$
be the associated Robin function. Likewise, let $G_{\mathcal{H}}$ be the Green
function for $\mathcal{H}$ and $\La_{\mathcal{H}}$ be the corresponding Robin
function.

\begin{prop}\label{conv-G_j}
For every $p \in \mathcal{H}$, $\big\{G_j(z,p)\big\}$ converges uniformly on compact
subsets of $\mathcal{H} \setminus \{p\}$ to $G_{\mathcal{H}}(z,p)$. In particular, $\La_j(p) \to \La_{\mathcal{H}}(p)$.
\end{prop}

The first step in proving this proposition is to localise the problem near the
point $p_0$ which is achieved by the following:

\begin{lem}\label{local}
There exists a neighbourhood $U$ of $p_0$ and a constant $R=R(D)$ such that
\[
0< G_{D}(z,p)-G_{U \cap D}(z,p)< 2\log\frac{2R+3\de(z)}{2R-\de(z)}
\]
for all $z,p \in U \cap D$ and where $\de(z)=d(z, \pa D)$.
\end{lem}
\begin{proof}
The first inequality is a consequence of the maximum principle and holds
for any neighbourhood of $p_0$. For the second inequality, choose a
neighbourhood $U$ of $p_0$ such that $\psi$ is defined on the closure of $U$. Then $U \cap D$
and $D$ share a common smooth piece of boundary, say $\ga$ containing $p_0$.
Choose and fix $R>0$ such that for each point $\z \in \ga$, it is possible to
draw a pair of balls each of radius $R$ and tangent to $\ga$ at $\z$ such that
one of them lies in $U \cap D$ (and hence in $D$) and the other lies outside $D$
(and hence outside $U \cap D$). If $U$ is sufficiently small, then the distance
$\de(z)$ of a point $z \in U \cap D$ to $\pa D$ is realized by a unique point
$\pi(z)
\in \ga$, i.e., $\de(z) = \big\vert z - \pi(z)\big\vert$. Let $z^{\prime}$ be the
`symmetric point' of $z$ with respect to $\ga$, i.e., $\pi(z)=(z+z^{\prime})/2$.
By Theorem 4.4 of \cite{GS}, we have for $z,p \in U \cap D$,
\begin{align*}
\log\left(1-\frac{2\de(z)}{2R+\de(z)}\right) < \log \left
\vert\frac{z^{\prime}-p}{z-p} \right \vert - G(z,p) <
\log\left(1+\frac{2\de(z)}{2R+\de(z)}\right)
\end{align*}
where $G=G_{U \cap D}$ or $G_{D}$. Therefore, $G_D(z,p) - G_{U \cap D}(z,p)$ is
bounded above by
\[
\log\left(1+\frac{2\de(z)}{2R+\de(z)}\right)-\log\left(1-\frac{2\de(z)}{
2R+\de(z)}\right) = 2\log\frac{2R+3\de(z)}{2R-\de(z)},
\]
which establishes the second inequality.
\end{proof}

Now choose $U$ and $R$ as in the above lemma where we may assume without loss of
generality that $U \cap D$ is simply connected. Also assume that $p_j \in U
\cap D$ for all $j$. Let $\ti D= U \cap D$ and $\ti D_j=T_j(U \cap D)$. By
similar arguments as earlier, the domains $\ti D_j$ converge to the half plane
$\mathcal{H}$ in (\ref{half-p}) in the Hausdorff sense. Let $\ti G_j$ be the
Green function for $\ti D_j$. In view of the above lemma, the convergence of
$G_j$ is controlled by the convergence of $\ti G_j$. Indeed, if $K$ is a compact
subset of $\mathcal{H}\setminus \{p\}$, then $K \subset \ti D_j \subset D_j$ and $T_j^{-1}(K) \subset U \cap D$ for all large $j$. Therefore, for all $z\in K$,
\[
0<G_{D}(T_j^{-1}z,T_j^{-1}p)-G_{\ti D}(T_j^{-1}z, T_j^{-1}p) < 2
\log\frac{2R+3\de(T_j^{-1}z)}{2R-\de(T_j^{-1}z)},
\]
which is equivalent to
\[
0<G_j(z,p)-\ti G_j(z,p) < 2 \log\frac{2R+3\de(T_j^{-1}z)}{2R-\de(T_j^{-1}z)}.
\]
Since
\[
T_j^{-1}z=p_j-\psi(p_j)z \to p_0
\]
uniformly on $K$, $\de(T_j^{-1}z) \to 0$ uniformly on $K$. It follows
that $G_j(z,p) - \ti G_j(z,p) \to 0$ uniformly on $K$. This means that
the first assertion of Proposition \ref{conv-G_j} would be proved if we have the following:

\begin{prop}
For every $p \in \mathcal{H}$, $\big\{\ti G_j (z,p)\big\}$ converges
uniformly on compact subsets of $\mathcal{H} \setminus \{p\}$ to
$G_{\mathcal{H}}(z,p)$.
\end{prop}

\begin{proof}
Since $\ti D$ is simply connected and $T_j$ is an affine map, $\ti D_j$ is also
simply connected. Choose a conformal map $\phi_j$ from $\ti D_j$ onto the unit
disc $\mathbb D$ such that $\phi_j(p)=0$. Then
\[
\ti G_j(z,p)= -\log \big\vert \phi_j(z) \big\vert.
\]
Also $\{\phi_j\}$ is a normal family and every limit map is defined on $\mathcal H$ since
$\ti D_j$ converges to $\mathcal{H}$ in the Hausdorff sense. Let $\phi$ be a limit map of some subsequence of
$\phi_j$. Note that $\phi(p) = 0$ as $\phi_j(p) = 0$ for all $j$. Furthermore, $\phi(\mathcal H) \subset \overline{\mathbb D}$.
But if $\big\vert \phi(a) \big\vert = 1$ for some $a \in \mathcal H$, the maximum principle implies that $\vert \phi \vert \equiv 1$ on $\mathcal H$.
This contradicts the fact that $\phi(p) = 0$. Hence $\phi : \mathcal H \to \mathbb D$. The inverses $\phi_j^{-1}$ too form a normal family on $\mathbb D$.
To see this, let $S \subset \mathbb C \setminus \overline{\mathcal H}$ be a compact set with more than $2$ points. Since the defining functions for $\ti D_j$ converge
to that for $\mathcal H$ in the $C^2$-topology on $S$, it follows that $S$ lies outside the closure of $\ti D_j$ for all large $j$. Hence the family $\{\phi_j^{-1}\}$ misses
at least $2$ points which implies normality. Let $\varphi : \mathbb D \to \overline{\mathcal H}$ be some limit point of this family. Note that $\varphi(0) = p$. Since $\mathcal H$ is conformally equivalent to $\mathbb D$, a similar argument involving the maximum principle shows that $\varphi (\mathbb D) \subset \mathcal H$. Furthermore, since
\[
\phi_j \circ \phi_j^{-1}(z) = \phi_j^{-1} \circ \phi_j(z) = z
\]
for each $j$, it follows that $\phi \circ \varphi(z) = \varphi \circ \phi(z) = z$. This shows that each limit point $\phi : \mathcal{H} \to \mathbb D$ is conformal, $\phi(p) = 0$ and therefore is the Riemann map for $\mathcal{H}$ up to a rotation. Therefore,
\[
\ti G_j(z,p) = -\log \big\vert \phi_j(z) \big\vert  \to - \log \big\vert \phi(z)\big\vert = G_{\mathcal{H}}(z,p)
\]
uniformly on compact subsets of $\mathcal{H} \setminus \{p\}$ as desired.
\end{proof}

\medskip

For the second assertion of Proposition \ref{conv-G_j}, let $B(p,r)$ be a disc that is compactly contained in $\mathcal{H}$. Then
it is compactly contained in $D_j$ for all large $j$. By (\ref{repn-La-1}),
\[
\La_j(p)=\log r + \frac{1}{2\pi }\int_{-\pi}^{\pi} G_j(p+ re^{i\theta}, p) \, d\theta.
\]
Since $G_j(z,p)$ converges uniformly on $\pa B(p, r)$ to $G_{\mathcal{H}}(z,p)$,
the integral above converges to
\[
\log r +  \frac{1}{2\pi }\int_{-\pi}^{\pi} G_{\mathcal{H}}(p+re^{i\theta}, p) \,
d\theta = \La_{\mathcal{H}}(p),
\]
as required.

By taking $p = 0 \in \mathcal H$ and using (1.2), it follows that
\[
\La_D(p_j) - \log\big(-\psi(p_j)\big) = \La_j(0) \to \La_{\mathcal H}(0) = -\log \big\vert \pa \psi(p_0) \big\vert,
\]
as $p_j \to p_0$ which proves the first claim in Theorem 1.3 (i). Exponentiating this gives
\begin{equation}
c_D(p_j)\big(-\psi(p_j)\big) \to \big\vert \pa \psi(p_0) \big\vert,
\end{equation}
as $p_j \to p_0$. Before calculating the
higher order boundary asymptotics of $c_D$, here is an observation that was also noted in \cite{GS}. This gives an alternate but equivalent formulation of (2.4).

\begin{cor}
Let $D$ be as in Theorem 1.3. Then for any $a \in D$
\begin{equation}\label{La-sinh}
\La(p_j)+\log\frac{\big\vert\pa G_D(p_j,a)\big\vert}{\text{\em sinh}
\, G_D(p_j,a)} \to 0
\end{equation}
as $p_j \to p_0 \in \pa D$.
\end{cor}
\begin{proof}
Note that
\begin{align*}
\La(p_j)+\log\frac{\big\vert \pa G_D(p_j,a) \big\vert}{\text{sinh} \, G_D(p_j,a)} = \big(\La(p_j)-\log
G_D(p_j,a)\big) + \log \big \vert\pa G_D(p_j,a) \big \vert + \log
\frac{G_D(p_j,a)}{\text{sinh} \, G_D(p_j,a)}.
\end{align*}
By taking $\psi = -G_D(z,a)$ as a smooth defining function for $\Ga$ near $p_0$ (the non-vanishing of the gradient of $G_D(z, a)$ follows from the Hopf lemma), the first
term converges to $-\log\big \vert \pa G_D(p_0,a)\big \vert$,
by the first assertion of Theorem \ref{c_D,c_n-asymp} (i) and the second term
converges to $\log\big\vert \pa G_D(p_0,a)\big \vert$, since $G_D(z, a)$ is smooth up to $\Gamma$. Finally, since $\sinh x/x \to 1$ as $x \to 0$,
the last term vanishes in the limit and hence (\ref{La-sinh}) follows.
\end{proof}

To calculate the boundary asymptotics of the derivatives of $\La_D$, note that
\begin{equation}\label{der-La_La_j}
\pa^{\al + \be}\La_D(p_j) \big(-\psi(p_j)\big)^{\al + \be} = \pa^{\al + \be}(0),
\end{equation}
and hence it is enough to study the convergence of the right hand side.

\begin{prop}\label{conv-H_j}
For every $p \in \mathcal{H}$ and $(\al, \be) \not= (0, 0)$
\[
\pa^{\al + \be}\La_j(p) \to \pa^{\al + \be}\La_{\mathcal{H}}(p),
\]
as $j \to \infty$.
\end{prop}

\begin{proof}
Let $H_j(z,p)$ and $H_{\mathcal{H}}(z, p)$ be the regular parts of the Green's function for $D_j$ and
$\mathcal{H}$ respectively. Note that if $(z,p) \in \mathcal {H} \times \mathcal{H}$ and $z
\neq p$, then by Proposition 2.4,
\[
H_j(z,p)=G_j(z,p) +\log \vert z - p \vert \to G_{\mathcal{H}}(z,p)+\log \vert z
- p \vert = H_{\mathcal{H}}(z,p).
\]
Also, if $z = p$, then $H_j(p,p)=\La_j(p) \to \La_{\mathcal{H}}(p)=H(p,p)$ by Proposition 2.4 again. Thus $H_{j}(z,p) \to H_{\mathcal{H}}(z,p)$ pointwise. Since
$H_j(z,p)$ is jointly harmonic, the mean value property implies that the convergence is also uniform on compact
subsets of $\mathcal{H} \times \mathcal{H}$.

\medskip

Let $B(q_0,r)$ be compactly contained in $\mathcal{H}$ such that $p \in B(q_0, r)$. Then this disc is
compactly contained in $D_j$ for all large $j$. By (2.2),
\[
\La_j(p)=\frac{1}{4\pi^2} \int_{[-\pi,\pi]^2}
H_j(q_0 + re^{i\theta},q_0 + re^{i\phi}) \frac{\big(r^2- \vert p - q_0\vert^2\big)^2}
{\vert q_0+re^{i\theta} - p \vert^2 \vert q_0+re^{i\phi} - p \vert^2}\, d\theta\,d\phi.
\]
Differentiating with respect to $p$ under the integral sign,
\[
\pa^{\al+\be}\La_j(p)=\frac{1}{4\pi^2} \int_{[-\pi,\pi]^2}
H_j(q_0 + re^{i\theta},q_0 + re^{i\phi}) \pa^{\al+\be} \frac{\big(r^2- \vert p - q_0\vert^2\big)^2}
{\vert q_0+re^{i\theta} - p \vert^2 \vert q_0+re^{i\phi} - p \vert^2}\, d\theta\,d\phi.
\]
Since $H_j(z,p)$ converges uniformly on
$\pa B(q_0,r) \times \pa B(q_0, r)$ to $H_{\mathcal{H}}(z,p)$, the above
integral converges to
\[
\frac{1}{4\pi^2} \int_{[-\pi,\pi]^2}
H_{\mathcal{H}}(q_0 + re^{i\theta},q_0 + re^{i\phi}) \pa^{\al+\be} \frac{\big(r^2- \vert p - q_0\vert^2\big)^2}
{\vert q_0+re^{i\theta} - p \vert^2 \vert q_0+re^{i\phi} - p \vert^2}\, d\theta\,d\phi = \pa^{\al+\be}\La_{\mathcal{H}}(p),
\]
as required.
\end{proof}

By taking $p = 0 \in \mathcal{H}$, this proposition shows that
\[
\pa^{\al+ \be}\La (p_j) \big(-\psi(p_j)\big)^{\al+\be} =\pa^{\al + \be}\La_j(0) \to \pa^{\al + \be}\La_{\mathcal{H}}(0),
\]
as $j \ra \infty$, and it remains to see that 
\[
\pa^{\al + \be}\La_{\mathcal{H}}(0) = -(\al+\be-1)!\big(\pa \psi(p_0)\big)^{\al} \big(\ov \pa \psi(p_0) \big)^{\be},
\]
by (1.2). This finishes the proof of Theorem 1.3 (i).

\medskip

Theorem 1.3 (ii) is a statement about the behavior of the coefficients $c_{n, D}$ near $\Ga$. It requires the following preliminary representation of these coefficients in terms of $H_D(z, p)$.

\begin{lem}\label{c_n-formula}
Let $D \subset \mathbb C$ be a regular domain. For every disc
$B(q_0,r)$ which is compactly contained in $D$,
\[
c_{n, D}(q)=\frac{1}{2\pi^2n!}\int_{[-\pi,\pi]^2} \pa^n H_D(q_0+re^{i\theta}, q_0+re^{i\phi}) \frac{(r^2-\vert q-q_0\vert^2)^2}{\vert
q_0+re^{i\theta}-q\vert^2 \vert q_0+re^{i\phi}-q \vert^2} d \theta d\phi, \quad
q \in B(q_0, r),
\]
where $\pa^n=\pa/\pa z^n$ is the derivative with respect to the first variable.
\end{lem}

\begin{proof}
We have
\[
c_{n, D}(p)=\frac{1}{n!}\pa^n h_D(p,p)=\frac{2}{n!}\pa^n H_D(p,p).
\]
For a fixed $p$, $H_D(z,p)$ is harmonic in $z$ and hence so is $\pa^n H_D(z,p)$.
For a fixed $z$, $H_D(z,p)$ is harmonic in $p$, and by its joint smoothness, it follows that
$\pa^n H_D(z,p)$ is harmonic in $p$ as well. A
repeated application of the Poisson integral formula to the function $\pa^n
H(z,p)$ together with Fubini's theorem completes the proof.
\end{proof}

\begin{prop}\label{conv-c_n}
For a fixed $p \in H$, let $c_{n, j}$ be the coefficients in the expansion of $h_{D_j}$ around $p$. Then for every $n \ge 1$ and non-negative integers $\al, \be$,
\[
\pa^{\al+\be}c_{n, j}(p) \to \pa^{\al+\be}c_{n, \mathcal{H}}(p),
\]
as $j \to \infty$.
\end{prop}
\begin{proof}
Let $B(q_0,r)$ be a disc that is compactly contained in $\mathcal{H}$ such that $p \in B(q_0, r)$.
This disc is then compactly contained in $D_j$ for all large $j$. We have by Lemma \ref{c_n-formula},
\[
c_{n, j}(p) = \frac{1}{2\pi^2n!}\int_{[-\pi,\pi]^2} \pa^n H_j (q_0+re^{i\theta}, q_0+re^{i\phi}) \frac{(r^2-\vert p-q_0\vert^2)^2}{\vert
q_0+re^{i\theta}-p\vert^2 \vert q_0+re^{i\phi}-p \vert^2} d \theta d\phi,
\]
and by differentiating with respect to $p$ under the integral sign,
\[
\pa^{\al+\be}c_{n, j}(p) = \frac{1}{2\pi^2n!}
\int_{[-\pi,\pi]^2} \pa^n H_j(q_0+re^{i\theta}, q_0+re^{i\phi}) \pa^{\al+\be}\frac{(r^2-\vert p-q_0\vert^2)^2}{\vert
q_0+re^{i\theta}-p\vert^2 \vert q_0+re^{i\phi}-p \vert^2} d \theta d\phi.
\]
Since $\pa^n H_j(z,p)$ converges uniformly on $\pa B(q_0, r) \times \pa B(q_0, r)$ to $\pa^n H_{\cal H}(z,p)$,
the above integral converges to
\[
\frac{1}{2\pi^2n!}\int_{[-\pi,\pi]^2} \pa^n H_{\mathcal{H}}(q_0+re^{i\theta}, q_0+re^{i\phi}) \pa^{\al+\be}\frac{(r^2-\vert p-q_0\vert^2)^2}{\vert
q_0+re^{i\theta}-p\vert^2 \vert q_0+re^{i\phi}-p \vert^2} d \theta d\phi =
\pa^{\al+\be}c_{n, \mathcal{H}}(p),
\]
as desired.
\end{proof}

To finish the proof of Theorem 1.3 (ii), note that
\[
H_D(z,p)=H_j\big(T_j(z), T_j(p)\big)+\log\big(-\psi(p_j)\big),
\]
which implies that
\[
c_{n, D}(p)=\frac{2}{n!}\pa^n H_D(p,p)= \frac{2}{n!} \pa^n H_j\big(T_j(p),T_j(p)\big)
\frac{1}{\big(-\psi(p_j)\big)^n} = c_{n, j}\big(T_j(p)\big)
\frac{1}{\big(-\psi(p_j)\big)^n}.
\]
By differentiating with respect to $p$,
\[
\pa^{\al+\be}c_{n, D}(p_j)\big(-\psi(p_j)\big)^{n+\al+\be}=\pa^{\al+\be}c_{n, j}(0).
\]
By Proposition 2.8, the right side converges to
\[
\pa^{\al+\be}c_{n, \mathcal{H}}(0) = -\frac{(n+\al+\be-1)!}{n!}\big(\pa
\psi(p_0)\big)^{n+\al}\big(\ov \pa \psi(p_0) \big)^{\be},
\]
from (1.4) as desired.


\subsection{The normalised Robin function} Let $D \subset \mathbb C$ be a smoothly bounded domain with smooth defining function $\psi$ which will be assumed to be defined on all of $\mbb C$. The normalised Robin function $\la$ for $D$ is defined as
\[
\la(p) = \La(p) - \log\big(-\psi(p)\big)
\]
and this is continuous up to $\ov D$ by Theorem 1.3 (i). Associated with each $p \in D$ is the affine map
\[
T_p(z) = \frac{z - p}{-\psi(p)}
\]
and if $D(p) = T_p(D)$, then $\la(p)$ is the Robin constant for the domain $D(p)$ at the origin. Note that $0 \in D(p)$ for all $p \in D$. This interpretation reveals that $\La$ satisfies stronger boundary asymptotics than those listed in Theorem 1.3. Before discussing this, note that the higher dimensional analogue of the normalised Robin function was studied by Levenberg-Yamaguchi. Locally, $p \mapsto D(p)$ is a smooth variation of domains given by (see \cite{LY})
\[
f(p,z) = 2 \Re  \left( \int_0^1 z \frac{\pa \psi}{\pa z}\big(p - \psi(p)tz\big) \; dt \right) - 1.
\]
By Hadamard's first variation formula (see for example \cite{MY}),
\[
\pa \la(p)=-\frac{1}{\pi}\int_{\pa D(p)} k_1(p,z) \left \vert \frac{\pa g}{\pa z}(p,z) \right\vert^2 \, ds,
\]
where
\[
k_1=\frac{\pa f}{\pa p}\bigg/\frac{\pa f}{\pa z},
\]
$g(p,z)$ is the Green function for $D(p)$ with pole at $0$, and $ds$ is the arc length measure (and not to be confused with the capacity metric as it will be clear from the context). Since $-g(p, z)$ is a defining function for $\pa D(p)$, a normal to $\pa D(p)$ is $-2\pa g /\pa \ov z$, whence $-2i\pa g/\pa \ov z$ is tangent to $\pa D(p)$. Now comparing $dz=z^{\prime}(t)dt$ and $ds=\big\vert z^{\prime}(t)\big\vert dt$, we have
\[
\left \vert \frac{\pa g}{\pa z} \right \vert \,ds= i\frac{\pa g}{\pa z} \,dz,
\]
along $\pa D(p)$, and so the variation formula can also be written as
\[
\pa \la(p)= -\frac{i}{\pi}\int_{\pa D(p)} k_1(p,z) \left \vert \frac{\pa g}{\pa z}(p,z) \right \vert \frac{\pa g}{\pa z}(p,z) \,dz.
\]

\begin{thm}\label{reg-norm-robin}
The normalised Robin function $\la \in C^1(\overline D)$.
\end{thm}

\begin{proof}
It suffices to show that for $p_0 \in \pa D$,
\[
\lim_{p \in D, \; p \ra p_0}\pa \la(p)
\]
exists. Without loss of generality, assume that $p_0=0$, $\pa \psi(p_0)=1$ and let $p_j \to 0$. The domains $D_j=D(p_j)$ converge to the half-plane
\[
\mathcal{H}=\{z: \Re z < 1/2\}.
\]
The M\"{o}bius transformation 
\[
w = M(z)=(z+1/2)\big/(z-3/2).
\]
maps this half-plane conformally onto the unit disc $\mbb D$ with $M(0) = -1/3$. Furthermore, the domains $\Om_j=M(D_j)$ converge to the unit disc and $M(0) = -1/3 \in \Om_j$ for all $j$.  For brevity, let $g = g(p, z)$. Then $g_j = g \circ M^{-1}(w)$ is the Green's function for $\Om_j$ with pole at $-1/3$ and by Proposition 2.4, it follows that $g_j \rightarrow g_{\mathbb D}$, the Green's function for the unit disc with pole at $-1/3$, uniformly on compact sets of $\mathbb D \setminus \{-1/3\}$. Writing $w=M(z)$ in the variation formula gives,
\begin{equation}
\pa \la(p_j)=-\frac{i}{\pi}\int_{\pa \Om_j}k_1\big(p_j,M^{-1}(w)\big) \left \vert \frac{\pa g_j}{\pa w}  \right \vert \frac{\pa g_j}{\pa w} \frac{dw}{\Big\vert \big(M^{-1}(w)\big)' \Big\vert}.
\end{equation}
To show that the integrals converge, observe first that
\[
M^{-1}(w) = (1/2 - 3w/2)/(w - 1),
\]
hence $\big\vert \big(M^{-1}(w)\big)' \big\vert = O\big( \vert w - 1 \vert^{-2}\big)$ near $w = 1$. By \cite{LY}, $\big\vert k_1(\cdot, z) \big\vert \lesssim \vert z \vert^2$ uniformly for all large $\vert z \vert$ and all $p_j$ close to $p_0$. This means that
\[
\Big\vert k_1\big(p_j, M^{-1}(w)\big) \Big\vert \lesssim \big\vert M^{-1}(w) \big\vert^2 = O\big( \vert w - 1 \vert^{-2}\big)
\]
near $w=1$. The domains $\pa \Om_j$ are all smoothly bounded and converge to $\pa \mathbb {D}$. In fact, $\Om_j$ is defined by
\[
\psi_j \circ M^{-1}(w) = -1 + 2 \Re\big(\partial \psi(p_j) M^{-1}(w)\big) + \psi(p_j) O\big(\vert M^{-1}(w) \vert^2\big)
\]
or equivalently by,
\begin{eqnarray*}
\rho_j(w) & = & \vert w - 1 \vert^2 \psi_j \circ M^{-1}(w)\\
                & = & - \vert w - 1 \vert^2 + 2 \Re \big( \psi(p_j) (1/2 - 3/2w)(\overline w - 1) \big) + \psi(p_j) \vert w - 1 \vert^2 O\big(\vert w -1 \vert^{-2}\big)
\end{eqnarray*}
after clearing denominators. The remainder term is thus rendered harmless near $w=1$. By working with their derivatives, it follows that $\rho_j$ converges to $\rho_{\infty}  = 1 - \vert w \vert^2$ in the $C^{\infty}$-topology in a fixed neighbourhood of $w = 1$, say $U$. Now split the integral in (2.7) as the sum of two integrals, one over $\partial \Om_j \setminus U$ and the other over $\partial \Om_j \cap U$. It remains to show that the derivatives of $g_j$ on $\partial \Om_j$ converge to the corresponding derivatives of $g_{\mathbb D}$ on $\partial \mathbb D$. This is a consequence of the Schauder estimates. Indeed, let $B$ be a small disc around $w = -1/3$ such that $B \subset \Om_j$ for all large $j$. Then the $C^2$-norm of $g_j$ on $\Om_j \setminus B$ is dominated by a constant times the sum of the $C^0$-norm of $g_j$ on $\pa \Om_j \setminus B$ and the $C^2$-norm of $g_j$ on $\pa B$. The constant appearing in this inequality is essentially harmless and depends on $\Om_j$ (which converge to $\mathbb D$) -- thus it can be chosen to be independent of $j$. Being harmonic on $\Om_j \setminus B$, the $C^0$-norm of $g_j$  is dominated by its $C^0$-norm on $\pa \Om_j \cup \pa B$. Note that $g_j = 0$ on $\pa \Om_j$ and therefore what matters is the behaviour of $g_j$ on $\pa B$. By Proposition 2.4, the $g_j$'s converge to $g_{\mathbb D}$ (along with all derivatives) away from $w = -1/3$ and hence their $C^2$-norms are bounded independent of $j$. Now let $a_j \in \pa \Om_j$ converge to $a \in \pa \mathbb D$. For a $\delta > 0$ (to be chosen later), consider the inward normals to $\pa \Om_j$ at $a_j$ of length $\delta$. Let $b_j \in \Om_j$ be such that the interval $[a_j, b_j]$ is normal to $\pa \Om_j$ at $a_j$ and has length $\delta$. Note that the $b_j$'s lie in a compact subset of $\mathbb D$. Let $b_0 \in \mathbb D$ be a limit point of $\{b_j\}$. Then
\[
\left \vert \frac{\pa g_j}{\pa w}(a_j) - \frac{\pa g_{\mathbb D}}{\pa w}(a) \right \vert \le \left \vert \frac{\pa g_j}{\pa w}(a_j) - \frac{\pa g_j}{\pa w}(b_j) \right \vert + \left \vert \frac{\pa g_j}{\pa w}(b_j) - \frac{\pa g_{\mathbb D}}{\pa w}(b_0) \right \vert  + \left \vert \frac{\pa g_{\mathbb D}}{\pa w}(b_0) - \frac{\pa g_{\mathbb D}}{\pa w}(a) \right \vert.
\]
The first and third terms are no more than a uniform constant times $\vert a_j - b_j \vert = \delta$ by the mean value theorem while the second one can be made arbitrarily small by Proposition 2.4 since the $b_j$'s are compactly contained in $\mathbb D$. Thus by choosing $\delta$ appropriately, it follows that the left side can be made arbitrarily small, i.e., the derivatives of $g_j$ on $\pa \Om_j$ converge to those of $g_{\mathbb D}$ on $\pa \mathbb D$ uniformly on a given compact neighbourhood of $\pa \mathbb D$. By working with the integrals on $\pa \Om_j \cap U$ and $\pa \Om_j \setminus U$ separately, it follows that the integrals in (2.7) converge to 
\[
-\frac{i}{\pi}\int_{\pa \mathbb D}k_1\big(p_0, M^{-1}(w)\big) \left \vert \frac{\pa g_{\mathbb D}}{\pa w}  \right \vert \frac{\pa g_{\mathbb D}}{\pa w} \frac{dw}{\Big\vert \big(M^{-1}(w)\big)' \Big\vert},
\]
and this completes the proof.
\end{proof}

A consequence of all this is that 
\[
c_D(z) \vert dz \vert = e^{-\Lambda_D(z)} \vert dz \vert  = \frac{e^{-\lambda_D(z)}}{\big(-\psi(z)\big)} \vert dz \vert \approx \frac{e^{-\lambda_D(z)}}{\text{dist}(z, \pa D)} \vert dz \vert.
\]
Since $\lambda_D \in C^1(\overline D)$, it follows that the capacity metric is uniformly comparable with a metric whose density is $\rho_D(z)$ (the hyperbolic density) times a $C^1$-smooth function on $\overline D$.

\section{More on the capacity metric}

\noindent In this section we will study the curvature and the boundary behaviour of the geodesics in this metric.  Recall that the curvature of $c_D(z) \vert dz \vert$ is
\[
\mathcal K(z)=-4 c^{-2}_D(z) \pa \overline \pa \log c_D(z).
\]

\subsection{Proof of Proposition 1.4} By (1.5),
\[
\big(-\psi(p)\big)^2 \; c_D^2(p) \to \big\vert \pa \psi(p_0) \big\vert^2,
\]
and by Theorem 1.3 (i),
\[
\big(-\psi(p)\big)^2 \;\pa \overline \pa \log c_D(p) \to \big\vert \pa \psi(p_0) \big\vert^2,
\]
as $p \to p_0$. Hence $\mathcal K(p) \to -4$ as $p \to p_0$.

\subsection{Proof of Theorem 1.5} 

For (i), it suffices to show that the capacity metric satisfies the hypotheses of Theorem 1.1 of [Herbort]. In fact, all that is needed is Property B of this theorem which essentially demands that $ds^2 = c^2_D(z) \vert dz \vert^2$ blow up at a certain rate near $\pa D$.  But this is immediate from the fact that for a nonzero $v \in \mathbb C$ at $z \in D$, 
\[
ds^2(z, v)/ \vert v \vert^2 = c^2_D(z)  \gtrsim \big(\psi(z)\big)^{-2} 
\]
from Theorem 1.3 (i). Thus, every nontrivial homotopy class in $\pi_1(D)$ contains a closed geodesic.

\medskip

\noindent For (ii), we need the following intermediary steps.

\medskip

\noindent {\sf Step 1:} For a smoothly bounded domain $D$, the capacity metric is $\delta$-hyperbolic in the sense of Gromov.

\medskip

By Theorem 1.3 (i) and a fortiori by the boundary behaviour of the normalised Robin function, there exists a constant $C > 1$ such that
\begin{equation}
C^{-1} \rho_D(z) \le c_D(z) \le C \rho_D(z)
\end{equation}
for all $z \in D$. By abuse of notation, the distance functions corresponding to $\rho_D$ and $c_D$ will again be denoted by the same symbols, i.e., for $a, b \in D$, the hyperbolic distance between them will be written as $\rho_D(a, b)$ while $c_D(a, b)$ will be the distance between them in the capacity metric. This will not cause any confusion for what is meant in each case will be clear from the context. The balls in these metrics will be denoted thus: $B^{\rho}_D(a, r)$ and $B^{c}_D(a, r)$ are the balls centered at $a \in D$ of radius $r$ in the hyperbolic and capacity metrics respectively. The subscript (in this case $D$) identifies the domain of these metrics.

\medskip

Recall the notion of $\delta$-hyperbolicity in the sense of Gromov: let $(X, d)$ be a metric space and $I = [a, b] \subset \mathbb R$ a compact interval. For $x, y \in X$, a map $\gamma : I \to X$ such that $\gamma(a) = x, \gamma(b) = y$ is called a geodesic segment if it is an isometry, i.e., $d\big(\gamma(s), \gamma(t)\big) = \vert s - t \vert$ for all $s, t \in \mathbb R$. Geodesic segments joining $x, y$ will be denoted by $[x, y]$ despite their possible non-uniqueness. The space $(X, d)$ is called a geodesic space if any pair of points $x, y \in X$ can be joined by a geodesic segment. A geodesic metric space $(X, d)$ is called $\delta$-hyperbolic (for some $\delta \ge 0$) if every geodesic triangle $[x, y] \cup [y, z] \cup [z, x]$ in $X$ is $\delta$-thin, i.e., 
\[
\text{dist}(w, [y, z] \cup [z, x]) < \delta 
\]
for all $w \in [x, y]$. Thus, geodesic triangles are thin in a coarse sense and $(X, d)$ behaves like a negatively curved manifold. Now (3.1) has two consquences namely, $c_D(z) \vert dz \vert$ is complete on $D$ since $\rho_D(z) \vert dz \vert$ is so and thus $(D, c_D)$ is a geodesic space and secondly, the identity map between the metric spaces $(D, \rho_D)$ and $(D, c_D)$ is a quasi-isometry. Since $D$ is smoothly bounded, it follows from \cite{RT} that $\rho_D$ is $\delta$-hyperbolic (for some $\delta$) and therefore so is $c_D$ for a possibly different $\delta$.

\medskip

\noindent {\sf Step 2:} Let $z(t)$ be a geodesic in the capacity metric that does not remain in a compact subset of $D$ for all $t \geq 0$. Then $z(t)$ approaches the boundary $\pa D$ at an exponential rate as $t \to +\infty$.

\medskip

Let $z : [0, +\infty) \to (D, c_D)$ be a geodesic. Then (3.1) implies that $z(t)$ is a quasi-geodesic in $(D, \rho_D)$ in the sense that
\begin{equation}
C^{-1} \vert s - t \vert \le \rho_D\big(z(s), z(t)\big) \le C \vert s - t \vert
\end{equation}
for all $s, t \in [0, +\infty)$. Suppose that $z(t)$ does not remain in a compact subset of $D$ for all $t \geq 0$. The first thing to do is to show that a geodesic that starts in a direction close to the normal to $\pa D$ does not deviate too much from it for all large $t$.

\begin{lem}\label{geo-est}
Let $z(t)$ be a unit speed geodesic in the capacity metric on
$D$ and  suppose for a point $p=z(t_0)$ near $\pa D$, we have
\[
0 < (\psi \circ z)^{\prime}(t_0) < +\infty.
\]
Then the estimates
\[
\text{\em (I)} \quad  -(\psi \circ z)(t) \lesssim (\psi \circ z)^{\prime}(t),
\]
and
\[
\text{\em (II)} \quad \left\vert \arg \Big(\pa \psi\big(z(t)\big) z^{\prime}(t)  \Big)\right\vert \leq \left\vert \arg \Big(z^{\prime}(t_0) \pa \psi\big(z(t_0)\big)\Big) \right\vert + C 
\]
are valid for $t_0 \leq t \leq t_0+100$. Here $C = C(t_0) \to 0$ as $t_0 \to +\infty$. 
\end{lem}

Geometrically, if at some point of time the angle between a geodesic and the normal to $\pa D$ at the nearest point is less than $\pi/2$, then (I) says that it remains so for some time and (II) gives an estimate of this angle. 

\begin{proof}
We divide the proof into two steps. In the first step, we verify (I) and (II) in a special case, namely when the domain is the unit disc, by explicit calculation. Then we prove the general case by localising the problem near the boundary and comparing the geodesics of $D$ with that of a one-sided neighbourhood of a boundary point conformally equivalent to the unit disc. Without loss of generality, assume that $t_0=0$.

\medskip

\noindent {\sf Step A:} For the unit disc $\mathbb D$ with defining function $\psi(z)=\vert z \vert -1$, let us look at the geodesics in the hyperbolic metric that start at $1/2<p<1$. The geodesics of $\mathbb D$ starting at the origin are of the form
\[
t \mapsto e^{i\theta}T(t), \quad T(t)=\frac{e^{2\al t}-1}{e^{2\al t}+1},
\]
whose initial velocity is $e^{i \theta} \alpha$. Therefore, the geodesics starting at $z(0)=p$ are given by
\[
t \mapsto z(t)=\frac{e^{i\theta}T(t)+p}{1+pe^{i\theta}T(t)}.
\]
Note that
\[
(\psi \circ z)^{\prime}=2\Re\big( \pa \psi(z) z^{\prime}\big) = \frac{1}{\vert z\vert}\Re(\ov z z^{\prime}),
\]
for $z\neq 0$, and
\[
 \ov z z^{\prime}=\frac{e^{-i\theta}T+p}{1+pe^{-i\theta}T} \cdot \frac{e^{i\theta}(1-p^2)T^{\prime}}{(1+pe^{i\theta}T)^2}
=\frac{(1-p^2)T^{\prime}(T+pe^{i\theta})}{\vert 1+pe^{i\theta}T\vert^2 (1+pe^{i\theta}T)}.
\]
We are given that $(\psi \circ z)^{\prime}(0) > 0$ which from the above calculation implies that
\[
(1-p^2)\al p\cos \theta > 0,
\]
and hence $\cos\theta > 0$. Now to prove (I), note that
\[
\frac{(\psi \circ z)^{\prime}}{-(\psi \circ z)} = \frac{1}{\vert z \vert(1-\vert z \vert)}\Re(\ov z z^{\prime}) =  \frac{1}{\vert z \vert(1-\vert z \vert)}\frac{(1-p^2)T^{\prime}(T+p^2T+p\cos\theta+p\cos\theta T^2)}{\vert 1+pe^{i\theta}T\vert^4 },
\]
where
\[
T'(t) = \frac{4 \al e^{2 \al t}}{(e^{2 \al t} + 1)^2},
\]
and $1 - \vert z(t) \vert$ (which is of the order of $1 - \vert z(t) \vert^2$) involves 
\[
1 - T^2(t) =  \frac{4 \al e^{2 \al t}}{(e^{2 \al t} + 1)^2},
\]
up to harmless universal constants that do not blow up near $\pa \mathbb D$. This proves (I) with a universal constant.

\medskip

For (II), note that
\begin{multline*}
\arg(\pa\psi(z)z^{\prime})=\arg(\ov z z^{\prime})=\arg \left( \frac{T+pe^{i\theta}}{1+pe^{i\theta}T} \right)
=\arg\big((T+pe^{i\theta})(1+pe^{-i\theta}T)\big)\\
= \tan^{-1} \left( \frac{p\sin\theta(1-T^2)}{T+p^2T+p\cos\theta+p\cos\theta T^2} \right),
\end{multline*}
and which is equal to $\theta$ for $t=0$. By the mean value theorem applied to $x \mapsto \tan^{-1}x$, 
\[
\left \vert \arg \left( \frac{T+pe^{i\theta}}{1+pe^{i\theta}T} \right) - \theta \right \vert \lesssim \big\vert T'(\tilde t) \big\vert 
\]
for some $\tilde t \in (0, 100)$. But then $\vert T' \vert \to 0$ as $t \to +\infty$ and thus in any interval of the form $[t_0, t_0 + 100]$, it follows that (II) holds.

\medskip

{\sf Step B:} Let $D$ be as in the statement of the Lemma. Let us first localize the problem near a boundary point $p_0 \in \pa D$. Choose
a neighbourhood $U$ of $p_0$ and a constant $R$ as in Lemma 2.3 and
without loss
of generality assume that $\ti D =U \cap D$ is simply connected. We will work
with geodesics $z(t)$ of the capacitymetric that start in $\ti D$. Corresponding to $z(t)$ 
consider the geodesic $\ti z(t)$ of the capacity metric in $\ti
D$ with initial conditions $\ti z(0)=z(0)=p$ and
$\ti{z}^{\prime}(0)=z^{\prime}(0)$. Since $\ti D$ is conformally
equivalent to the unit disc and also shares a smooth piece
of boundary $\Ga$ with $D$ which is defined by $\psi$, (I) and (II) hold for $\ti z$:
\begin{equation}\label{ti-z}
(\psi \circ \ti z)^{\prime}(t) \gtrsim -(\psi \circ \ti z) (t),
\end{equation}
and
\begin{equation}\label{ti-z-prime}
\left\vert \arg \Big(\pa \psi\big(\ti z(t)\big) \ti z^{\prime}(t)  \Big)\right\vert \leq \left\vert \arg \Big(\ti z^{\prime}(0) \pa \psi\big(p\big)\Big) \right\vert + C
\end{equation}
for $0 \leq t \leq 100$. The next step is to show that $z(t)$ and $\ti z(t)$ remain close to each other with nearly equal speed for some amount of time and deduce (I) and (II) for $z$ from the corresponding results for $\ti z$. First note that since the
geodesic $z(t)$ has unit capacity speed, i.e,
\[
c_D\big(z(t)\big)\vert z^{\prime}(t)\vert=1,
\]
for all $t$, Theorem 1.3 shows that,
\begin{equation}\label{bdd_speed} 
\big\vert z^{\prime}(t)\big\vert \approx -\psi\big(z(t)\big) \approx \de\big(z(t)\big).
\end{equation}
Now consider the transformation
\[
w=T(z)=\frac{z-p}{\de(p)}
\]
and let $\Om=T(D)$ and $\ti \Om = T(\ti D)$. Then for $w,q \in \ti \Om$,
\[
0< G_{\Om}(w, q) - G_{\ti \Om}(w, q) < 2 \log \frac{2R +
3\de(T^{-1}w)}{2R-\de(T^{-1}w)},
\]
and hence
\[
0< \La_{\Om}(q)-\La_{\ti \Om}(q) < 2 \log \frac{2R +
3\de(T^{-1}q)}{2R-\de(T^{-1}q)}.
\]
This implies that on a sufficiently small neighbourhood $N$ of the origin
that is compactly contained in $\ti \Om$, we have
\begin{equation}\label{perturb}
c_{\ti \Om}(q) \geq c_{\Om}(q) \geq c_{\ti\Om}(q)-\de(p) O(1).
\end{equation}
From standard perturbation results for ordinary differential
equations, the geodesics $w(t)=T\big(z(t)\big)$ and $\ti w(t)=T\big(\ti
z(t)\big)$ satisfy
\[
\big\vert \ti w (t)-w (t) \big\vert \lesssim \de(p) \quad \text{and} \quad
\big\vert
\ti w^{\prime}(t) - w^{\prime}(t) \big\vert \lesssim \de(p),
\]
for $\vert t\vert \leq 500$ as $\vert w^{\prime}(0) \vert \lesssim 1$ by
(3.5). Converting back from $w$ to our original
coordinate $z$, we obtain
\[
\big\vert \ti z (t)-z (t) \big\vert \lesssim \de^2(p)  \quad \text{and}
\quad \big\vert \ti z^{\prime} (t) - z^{\prime} (t) \big\vert \lesssim \de^2(p)
\]
for  $\vert t\vert \leq 500$. Thus these two geodesics are close to each other and have nearly equal speed for this period of time. The above estimates together with the fact that $\psi$ is
smooth imply that
\begin{equation}\label{z-tiz}
\big\vert z^{\prime}\pa\psi(z)-\ti z^{\prime}\pa\psi(\ti z)\big\vert \leq \big\vert z^{\prime} \pa\psi(z) - \ti z^{\prime}\pa \psi(z)\big\vert + \big\vert  \ti z^{\prime}\pa \psi(z) - \ti z^{\prime} \pa \psi(\ti z)\big\vert 
\lesssim \de^2(p),
\end{equation}
for $\vert t\vert \leq 500$.

\medskip

Now to prove (I), note that $(\psi \circ z)^{\prime}=2\Re(z^{\prime} \pa \psi)$, and so by (3.7),
\[
\big\vert (\psi \circ z)^{\prime}(t) - (\psi \circ \ti z)^{\prime}(t) \big\vert \lesssim
\de^2(p).
\]
Combining with (3.3),
\[
(\psi \circ z)^{\prime} \geq (\psi \circ \ti z)^{\prime} + \ti C \de^2(p) \geq C (-\psi \circ \ti z) + \ti C \de^2(p) \geq C \de(p) + \ti C \de^2(p)
\geq C \de(p) \gtrsim (\de \circ z) (t) \gtrsim
-(\psi \circ z)(t),
\]
for $0 \leq t \leq 100$ as required. Here it is important to note that the constant $C$ comes from (I) and is hence universal. For (II), combining (3.4) and (3.7), and (again!) using $C$ as a different universal constant,
\[
\Big\vert \arg\big(z^{\prime} \pa \psi(z)\big)\Big\vert
= \Big\vert \arg\big(\ti z^{\prime} \pa \psi (\ti z)\big)\Big\vert 
\leq \Big\vert \arg \big(\ti z^{\prime}(0) \pa \psi(p)\big) \Big\vert + C 
= \Big\vert \arg \big(z^{\prime}(0) \pa \psi(p)\big)\Big\vert + C,
\]
for $0 \leq t \leq 100$ as required. Here, the first and last equalities come from (3.7) which shows that the magnitudes of $z^{\prime}\pa\psi(z)$ and $\ti z^{\prime}\pa\psi(\ti z)$ are essentially the same and hence their arguments must be the same.
\end{proof}

\begin{lem}
Let $z(t)$ be a unit speed geodesic not remaining in a compact subset of $D$ for
all $t \geq 0$. There is a time $t_0>0$ and a constant $C>0$ such that
\begin{equation}\label{lim-geo}
\frac{\psi\big(z(t)\big)}{\psi\big(z(t_0)\big)} \leq e^{-C(t-t_0)} \quad \text{for $t>t_0$ large}.
\end{equation}
\end{lem}

\begin{proof}
Let $\ep>0$ be small and let $t_0=\min\Big\{t \geq 0: \psi\big(z(t)\big) \geq - \ep\Big\}$. Then
\begin{itemize}
\item [(i)] $z(t_0)$ is near $\pa D$, and
\item [(ii)] $(\psi \circ z)^{\prime}(t_0)\geq 0$.
\end{itemize}
The conclusion of the previous lemma is that once a geodesic is sufficiently close to $\pa D$ and has positive speed, then it continues to move towards $\pa D$ for a fixed interval of time; its speed does not decay (by (I)) and its direction remains essentially the same throughout this interval (by (II)). Since the constants entering into (I) and (II) have the listed properties, a connectedness argument now shows that the geodesic retains these two properties for all large $t$. Hence $\psi\big(z(t)\big)$ is a increasing function for all large $t$. For this reason, it follows that eventually 
\[
(\psi \circ z)^{\prime}(t) \gtrsim -(\psi \circ z)(t).
\]
Also, from (3.5),
\[
(\psi \circ z)^{\prime}(t) \lesssim -(\psi \circ z)(t).
\]
Thus,
\begin{equation}\label{bev-geo}
(\psi \circ z)^{\prime}(t) \approx -(\psi \circ z)(t),
\end{equation}
from which (3.8) follows upon integration.
\end{proof}

\medskip

\noindent {\sf Step 3:} The geodesic $z(t)$ hits the boundary $\pa D$ at a unique point.

\medskip

Consider the sequence of points $z(n)$ which satisfy
\[
C^{-1} \leq \rho_D\big(z(n), z(n+1)\big) \leq C
\]
for all $n \geq 1$ by (3.2) -- this means that
\begin{equation}
z(n + 1) \in B^{\rho}_D\big(z(n), C\big).
\end{equation}
We will need the following localization lemma for the hyperbolic metric which says that the length of a vector based at a point close to, say $\zeta \in\pa D$ is essentially the same when measured in either the hyperbolic metric on $D$ or $U \cap D$, where $U$ is a beighbourhood of $\zeta$.  

\begin{lem}
For every $\zeta \in \pa D$, there exist a pair of arbitrarily small euclidean neighbourhoods $\zeta \in V \subset U$ of a uniform size and a uniform constant $C = C(U, V) > 0$ such that 
\[
B^{\rho}_{U \cap D}(p, \eta) \subset B^{\rho}_D(p, \eta) \subset B^{\rho}_{U \cap D}(p, C \eta)
\]
for every $p \in V \cap D$ and every $\eta > 0$. 
\end{lem}

This follows from the existence of peak functions at points of $\pa D$ -- use the Riemann mapping theorem to identify a one-sided neighbourhood of a boundary point of $\pa D$ with the unit disc. A proof of this series of inclusions for balls in the Kobayashi metric near strongly pseudoconvex points is well known. The same steps can be applied in this case as well. The other ingredient is an estimate on the size of the hyperbolic ball $B^{\rho}_{U \cap D}(p, R)$ in terms of the euclidean distance between $p$ and the boundary of $U \cap D$.

\begin{lem}
For $\alpha \in \mathbb D$ and $r > 0$, the hyperbolic ball
\[
B^{\rho}_{\mathbb D}(\alpha, r) \subset B\big(\alpha, C \rm{dist}(\alpha, \pa \mathbb D)\big)
\]
for some $C = C(r) > 0$.
\end{lem}

\begin{proof}
For $\alpha \in \mathbb D$, let 
\[
\phi(z, \alpha) = \frac{z - \alpha}{1 - \overline \alpha z}.
\]
Then
\[
B^{\rho}_{\mathbb D}(\alpha, r) = \left\{ z \in \mathbb D: \vert \phi(z, \alpha) \vert < \frac{e^{2r} -1}{e^{2r} + 1} \right\}
\]
which is seen to be a euclidean disc whose center and radius are
\[
A = \frac{(1- \eta^2) \alpha}{1-{\eta}^2 \vert \alpha \vert^2}, \; \; R = \frac{\eta(1 - \vert \alpha \vert^2)}{1 - {\eta}^2 \vert \alpha \vert^2}
\]
respectively, where 
\[
\eta =  \frac{e^{2r} -1}{e^{2r} + 1}.
\]
For $\vert \tau \vert \leq 1$ and $\theta \in \mathbb R$,
\[
\vert A + R \tau e^{i \theta}  - \alpha \vert \leq \vert A - \alpha \vert + R
\]
and by using the above expressions for $A, R$, it can be seen that both terms are of the order of $1 - \vert  \alpha \vert$. The constants that appear only depend on $\eta$ and hence only on $r$.
\end{proof}

Fix $N$ large enough so that $z(N)$ is close enough to a boundary point, say $\zeta \in \pa D$. By (3.10) and the above lemmas,
\[
z(N+1) \in B^{\rho}_D\big(z(N), C\big) \subset B^{\rho}_{U \cap D}\big(z(N), \tilde C\big) \subset B\big(z(N), C^{\ast} {\rm{dist}}(z(N), \pa D)\big),
\]
where the constants are independent of the points $z(n)$. Thus
\[
\big\vert  z(N+1) - z(N) \big\vert \lesssim {\rm{dist}}\big(z(N), \pa D\big) \lesssim e^{-N},
\]
where the last inequality comes from Step $2$. By repeating these steps, it follows that
\[
\big\vert  z(n+1) - z(n) \big\vert \lesssim {\rm{dist}}(z(n), \pa D) \lesssim e^{-n},
\]
for all $n \geq N$. Hence $z(n)$ converges to a unique point on $\pa D$. Finally, thanks to (3.2) once again, not only does (3.10) hold as indicated but the geodesic segment 
\[
z\big([n, n+1]\big) = \big\{ z(t): n \leq t \leq n+1\big\}
\]
is also contained in $B^{\rho}_D\big(z(n), C\big)$ for all large $n$. It follows that $z(t)$ converges to a unique boundary point as $t \to +\infty$.

\medskip

For Theorem 1.5 (iii), recall the notion of a geodesic loop from \cite{He} -- a geodesic loop in a Riemannian manifold $(M, g)$ based at $x \in M$ consists of a nonconstant geodesic
$\gamma : \mathbb R \to M$ and times $t_1, t_2 \in \mathbb R$ with $t_1 < t_2$ such that $\gamma(t_1) = \gamma(t_2) = x$. Thus the geodesic segment $\gamma : [t_1, t_2] \to M$ defines a loop in $M$ based at $x$. The main idea is to use Lemma $6$ of \cite{He} which says that if $(M, g)$ is a complete Riemannian manifold whose universal cover is infinitely sheeted and $x_0 \in M$ is a point through which no closed geodesic passes and $K \subset M$ is a compact set that contains all possible geodesic loops through $x_0$, then there is a geodesic spiral passing through $x_0$. Thus the problem reduces to finding such a compact subset $K$ and this is addressed in the following:

\begin{prop}
There exists an $\ep>0$ such that for each geodesic $\ga$ with $\psi\circ \ga(0)>-\ep$ and $(\psi \circ \ga)^{\prime}(0)=0$, it follows that $(\psi \circ \ga)^{\prime\prime}(0)>0$.
\end{prop}

Suppose that $z_0 \in D$ is such that no closed geodesic passes through it and let $\psi$ be a smooth defining function for $\pa D$. Take this $\ep > 0$ and let $\ep_1 = \min\{\ep, -\psi(z_0) \}$. Then 
\[ 
K = \big\{ z \in D : \psi(z) \leq -\ep_1\big\}
\]
is the compact set that works. Indeed, let $\gamma :[t_1. t_2] \to D$ be a geodesic loop based at $z_0$ and suppose that it does not lie in $K$. Then $\gamma$ enters the $\ep_1$-band around $\pa D$ and being a loop, it must turn back and hence $\psi \circ \gamma$ must have a maximum somewhere, say at $t_0 \in (t_1, t_2)$. This implies that $(\psi \circ \gamma)(t_0) > -\ep, (\psi \circ \gamma)^{\prime}(t_0) = 0$ and $(\psi \circ \gamma)^{\prime \prime}(t_0) < 0$. But this contradicts the above proposition.

\begin{proof}
If possible, assume that this is not true. Then there exists a sequence of geodesics $c_{\nu}$ such that:
\begin{itemize}
\item [(i)] $a_{\nu}=c_{\nu}(0)$ converges to a point $a_0 \in \pa D$,
\item [(ii)] $(\psi \circ c_{\nu})^{\prime}(0)=0$
\item [(ii)] $(\psi \circ c_{\nu})^{\prime\prime}(0)\leq 0$.
\end{itemize}
Without loss of generality let $a_0=0$ and $\nabla\psi(0)=2(\pa \psi/\pa \ov z)(0)=1$.
For sufficiently large $\nu$, the distance between $a_{\nu}$ and $\pa D$, say $\de_{\nu}$, is realised by a unique point $\pi(a_{\nu}) \in \pa D$. Apply translation and rotations to $D$ to obtain $D_{\nu}$ with defining functions $\psi_{\nu}$ such that
\begin{itemize}
\item $\pi(a_{\nu})$ corresponds to $0$ and $\pa \psi_{\nu}(0)=1$
\item The geodesic $c_{\nu}$ corresponds to $\ga_{\nu}$ that has the following properties:
\begin{itemize}
\item [(a)] $p_{\nu}=\ga_{\nu}(0)=-\de_{\nu}$
\item [(b)] $(\psi_{\nu} \circ \ga_{\nu})^{\prime}(0)=0$
\item [(c)] $(\psi_{\nu}\circ \ga_{\nu})^{\prime\prime}(0)\leq 0$
\end{itemize}
\end{itemize}
Note that
\[
(\psi \circ \ga)^{\prime\prime} = 2\Re \left(\pa \psi \ga^{\prime\prime}\right) + 2 \Re\left( \pa\pa\psi{\ga^{\prime}}^2\right)+2\pa\ov\pa\psi \vert \ga^{\prime} \vert^2,
\]
and
\[
\ga^{\prime\prime}=\pa \La_D {\ga^{\prime}}^2,
\]
so that by $(c)$,
\[
2\Re \big(\pa \psi_{\nu}(p_{\nu}) \pa \La_{\nu}(p_{\nu}) {\ga_{\nu}^{\prime}}(0)^2\big) + 2 \Re\big(\pa\pa\psi_{\nu}(p_{\nu}){\ga_{\nu}^{\prime}}(0)^2\big)+2\pa\ov\pa\psi_{\nu}(p_{\nu}) \big\vert \ga_{\nu}^{\prime}(0) \big\vert^2 \leq 0.
\]
Dividing throughout by $\vert \ga_{\nu}^{\prime}(0)\vert^2$, we may assume that $\ga_{\nu}^{\prime}(0)$ is a unit vector. Let the limit (after passing to a subsequence if necessary ) of $\ga_{\nu}^{\prime}(0)$ be the unit vector $v$.
Multiplying by $-\psi_{\nu}(p_{\nu})$ (which is positive) and taking limit $\nu \to \infty$, we get
\[
\lim_{\nu \to \infty}2\Re \Big(\pa \psi_{\nu}(p_{\nu})\pa \La_{\nu}(p_{\nu}) \big(-\psi_{\nu}(p_{\nu})\big) {\ga_{\nu}^{\prime}}(0)^2\Big)  \leq 0,
\]
as the last two terms go to $0$. This, after using Theorme 1.3 (ii), implies that
\[
2\Re\Big(\pa \psi(0)\big(-\pa\psi(0)\big)v^2\Big) \leq 0
\]
i.e.,
\begin{equation}\label{v-1}
\Re v^2 \geq 0.
\end{equation}
On the other hand (b) implies that
\[
2 \Re \big(\pa \psi_{\nu} (p_{\nu})\ga_{\nu}^{\prime}(0)\big)=0.
\]
Letting $\nu \to \infty$,
\[
2 \Re\big(\pa \psi(0) v \big)=0,
\]
which gives
\begin{equation}\label{v-2}
\Re v =0.
\end{equation}
Now (3.11) and (3.12) imply that $v=0$ which contradicts the fact that $v$ is a unit vector and this proves the proposition.
\end{proof}


For the proof of Corollary 1.6, note that the affine connection corresponding to
the capacity metric is given by
\[
\nabla \frac{\pa}{\pa z}=\pa \log c_D (z)\,dz \otimes
\frac{\pa}{\pa z} = -\pa \La_D(z)\,dz \otimes
\frac{\pa}{\pa z}.
\]
Therefore (see for example the discussion in \cite{GS}) the euclidean curvature 
of a geodesic $z(s)$ parametrized by euclidean arc-length $s$ with unit tangent vector 
$z^{\prime}(s)$ is
\begin{equation}\label{curv-geod}
\kappa\big(z(s)\big)=\Im \Big(\pa \La_D \big(z(s)\big)z^{\prime}(s)\Big).
\end{equation}
By Theorem 1.3 (i) it can be seen that
\[
\bigg\vert \kappa\big(z(s)\big)\Big(-\psi\big(z(s)\big)\Big) \bigg\vert \approx 1
\]
which completes the proof.


\subsection{Proof of Theorem 1.7} 

First consider the case $k=p+q \neq 1$. Since $\mathcal{H}^0_2(D)\cong \mathcal{H}^2_2(D)$, it is enough to show that there is no nonzero square integrable harmonic function on $D$ with respect to $ds =  c_D \vert dz \vert$. But since $ds$ is complete and K\"{a}her, any such function is constant, see for instance \cite{Yau}. Moreover, since $D$ has infinite volume with respect to $ds$:
\[
\int_{D} \frac{i}{2} c^2_D(z) \, dz \wedge d\ov z \gtrsim \int_{D} \frac{1}{d(z,\pa D)^2} = \infty,
\]
such a function must be $0$.

For the case $k=p+q=1$, we will prove that
\begin{equation}
ds^2 \approx (-\psi)^{-1} \vert dz\vert^2 + (-\psi)^{-2}\vert \pa \psi\vert^2 \vert dz\vert^2,
\end{equation}
uniformly near $\pa D$ where $\psi$ is a smooth defining function for $D$. The infinite dimensionality of $\mathcal{H}^{p,q}_2(D)$ will then follow from \cite{Oh}. Let us denote the right hand side of  (3.14) by $dt^2$.  If $z_0 \in \pa D$ and $v$ is any nonzero complex vector, then using Theorem 1.3, 
\begin{align*}
\lim_{z \to z_0} \frac{ds^2_z(v,v)}{dt^2_z(v,v)} & = \lim_{z \to z_0} \frac{c_{D}(z)^2}{\big(-\psi(z)\big)^{-1} + \big(-\psi(z)\big)^{-2} \big\vert \pa \psi(z)\big\vert^2} \\
&= \lim_{z\to z_0} \frac{\big(-\psi(z)\big)^2 c_{D}(z)^2}{ -\psi(z) + \big\vert \pa \psi (z) \big\vert^2 }\\
& = \frac{\vert \pa \psi (z_0) \vert^2}{\vert \pa \psi (z_0) \vert^2 }\\
&=1.
\end{align*}
Therefore, the ratio
$ds^2_z/dt^2_z$
is uniformly bounded above and below by positive constants near $z_0$ and hence near $\pa D$ by compactness and this completes the proof.


\section{Critical points of the Green's function : Proof of Theorem 1.8}

\noindent We will need uniform estimates for the Green's functions, away from the diagonal, of a varying family of domains.

\begin{prop}
Let $D \subset \mathbb C$ be a smoothly bounded domain and $D_{k} \subset \mathbb C$
a sequence of smoothly bounded domains converging to it in the 
$C^{\infty}$-topology. Fix $z_0 \in D$ and let $B = B(z_0,r)$ be a small disc around it such that
$\ov B \subset D_k$ for large $k$. Let $z_k \in B$ converge to $z_0$. Then for every $\epsilon > 0$ 
there exists a constant $C > 0$ such that
\[
\Vert \partial^{\alpha + \beta} G_k(z_k,\cdot) \Vert_{C^2(D_k \setminus B)} \leq C \; \Vert \partial^{\alpha + \beta}G(z_0, \cdot) \Vert_{C^2(D \setminus B)} + \epsilon,
\]
for $k$  large, where $G_k, G$ are the Green's functions for $D_k, D$ respectively and the derivatives $\partial^{\alpha + \beta}$ are taken with respect to the first variable. In particular, if $(z_k, \zeta_k) \in D_k\times D_k$ converges to $(z_0, \zeta_0) \in D_0\times \pa D_0$, then,
\[
\lim_{k \rightarrow \infty} \partial^{\alpha + \beta} G_k(z_k,\zeta_k) = \partial^{\alpha + \beta}G(z_0, \zeta_0).
\]

\end{prop}

\begin{proof}
Let $h_k(\zeta) = \partial^{\alpha + \beta}G_k(z_k, \zeta)$ and
$h(\zeta) = \partial^{\alpha + \beta} G(z_0, \zeta)$ and define
$\phi_k(\zeta) = \eta(\zeta)h_k(\zeta)$ and $\phi(\zeta) = \eta(\zeta) h(\zeta)$ where $\eta$ is compactly supported near $\partial B$ and $\eta \equiv 1$ on $\partial B$. 
The Schauder estimates give
\[
\Vert h_k \Vert_{C^2(D_k \setminus B)} \le C_k \left(\Vert h_k \Vert_{C^0(D_k \setminus B)} + \Vert \phi_k \Vert_{C^2(D_k \setminus B)} \right).
\]
By Proposition 2.2, it follows that $G_k(z_k, \zeta) \rightarrow G(z_0, \zeta)$ uniformly, if $\zeta$ varies over compact sets that do not intersect $B$. Since $\eta$ is compactly supported, we have 
\[
\Vert \phi_k \Vert_{C^2(D_k \setminus B)} \le \Vert \phi \Vert_{C^2(D \setminus B)} + \epsilon \le C \; \Vert h \Vert_{C^2(D\setminus B)} + \epsilon,
\]
for all $k$ large. 

To estimate $\Vert h_k \Vert_{C^0(D_k \setminus B)}$, note that $G_k(z, \zeta) \equiv 0$ as a function of $z$ for each $\zeta \in \partial D_k$. Hence $h_k (\zeta) = \partial^{\alpha + \beta}G_k(z_k, \zeta) = 0$ for all $\zeta \in \partial D_k$. Since $h_k$ is harmonic in $D_k$, the maximum principle shows that $\Vert h_k \Vert_{C^0(D_k \setminus B)}$ is dominated by its $C^0$-norm on $\partial B$. But then by appealing to Proposition 2.2 again, 
\[
\Vert h_k \Vert_{C^0(D_k \setminus B)} \le C \Vert h \Vert_{C^0(D \setminus B)} + \epsilon < C \Vert h \Vert_{C^2(D \setminus B)} + \epsilon.
\]
It remains to note that the constant $C_k$ in the Schauder estimate above depends on the $D_k$'s, which vary smoothly. Hence the $C_k$'s are uniformly bounded.

\medskip

Fix $\delta >0$. Let $p_k \in D_k$ be such that $d(p_k, \partial D_k) = \delta$
along an inward pointing normal to $D_k$ which passes through $\zeta_k$. Then $p = \lim_{k\rightarrow \infty} p_k$ is at a distance $\delta$ on the normal to $\pa D$ at $\zeta_0$. Observe that
\begin{multline*}
\abs{\pa^{\alpha + \beta} G_k(z_k, \zeta_k)  -
	\pa^{\alpha + \beta} G(z_0, \zeta_0)}  \leq
	\abs{\pa^{\alpha + \beta} G_k(z_k, \zeta_k)  -
	\pa^{\alpha + \beta} G_k(z_k, p_k)} \\+
	\abs{\pa^{\alpha + \beta} G_k(z_k, p_k)  -
	\pa^{\alpha + \beta} G_k(z_k, p)}  
	+ \abs{\pa^{\alpha + \beta} G_k(z_k,p)  -
	\pa^{\alpha + \beta} G(z_0, p)}\\ +
	\abs{\pa^{\alpha + \beta} G(z_0,p)  -
	\pa^{\alpha + \beta} G(z_0, \zeta_0)}.
\end{multline*}
The two middle terms can be made as small as possible for large
$k$ as the Green's functions converge uniformly by Proposition 2.2. The previous bound on the derivatives of $G_k$ near $\pa D_k$ implies that the first and last terms are dominated by a harmless constant times $\abs{\zeta_k  - p_k} = \delta$. The result follows since $\delta$ is arbitrary.
\end{proof}

\noindent {\it Proof of Theorem 1.8}: Let $a \in \pa D$, $a \not=\zeta_0$ be such that
$\partial G(a, \zeta) \neq 0$
for all $\zeta \in D$. Let $a_{k}\in D_{k}$
converge to $a$. For large $k$, 
$\partial G_k(a_k, \zeta) \neq 0$
for every $\zeta \in D_{k}$. Suppose there exists a sequence 
$(z_k,\zeta_k) \in D_k \times D_k$ converging to $(z_{0}, \zeta_{0})$
such that $\partial G_{k}(z_{k},\zeta_{k}) = 0$. We first prove that $K(z_0,\zeta_0) = 0$.

\medskip

\noindent Following \cite{GS}, consider
\[
F_k(z,\zeta) = \frac{\partial G_k(z,\zeta)}{\partial G_k(a,\zeta)}, \quad z,\zeta \in D_k.
\]
Consider the point
$\eta_k \in \partial D_k$ such that $d(\zeta_k, \pa D_k) = \abs{\zeta_k -\eta_k}$. Observe that
$\pa G_k(z_k,\eta) = 0$ for $\eta \in \pa D_k$. Differentiating this by the chain rule gives
\begin{align*}
\frac{\pa^2 G_k}{\pa z \pa \zeta}(z_k,\eta_k)T(\eta_k) +
\frac{\pa^2 G_k}{\pa z \pa \ov \zeta}(z_k,\eta_k)\ov{T(\eta_k)} = 0,
\end{align*}
where $T(\eta_k)$ is the unit tangent vector to $\pa D_k$ at $\eta_k$.  Now consider the Taylor series expansion of
$\pa G_k(z_k,\zeta)$ around $\eta_k$:

\begin{align*}
\pa G_k(z_k,\zeta) &=
	\frac{\pa^2 G_k}{\pa z \pa \zeta}(z_k,\eta_k)(\zeta - \eta_k) +
		\frac{\pa^2 G_k}{\pa z \pa \ov \zeta}(z_k,\eta_k)
		(\ov\zeta - \ov\eta_k) + O\big(\abs{\zeta  - \eta_k}^2\big).
\end{align*}

Substituting for
$\frac{\pa^2 G_k}{\pa z \pa \zeta}(z_k,\eta_k)$ from above, we get
\begin{align*}
\pa G_k (z_k,\zeta_k)
	&=\frac{\pa^2 G_k}{\pa z \pa \ov \zeta}(z_k,\eta_k)
			\left(T(\eta_k) (\ov\zeta_k - \ov\eta_k)- \ov{T(\eta_k)}
				(\zeta_k - \eta_k) \right)\ov{T(\eta_k)} +
				O(\abs{\zeta_k - \eta_k}^2) \\
	&= 2i\frac{\pa^2 G_k}{\pa z \pa \ov \zeta}(z_k,\eta_k)
				\ov{T(\eta_k)}\Im\big((\ov\zeta_k - \ov\eta_k)
					T(\eta_k)\big) + O\big(\abs{\zeta_k - \eta_k}^2\big) \\
	&= 2i\frac{\pa^2 G_k}{\pa z \pa \ov \zeta}(z_k,\eta_k)
				\ov{T(\eta_k)}\abs{\zeta_k-\eta_k} +
				O\big(\abs{\zeta_k - \eta_k}^2\big).
\end{align*}
With these observations, 
\begin{align*}
F_k(z_k,\zeta_k)
	&= \frac{\partial G_k(z_k,\zeta_k)}{\partial G_k(a,\zeta_k)}, \quad z,\zeta \in D_k \\
	&= \frac{2i\frac{\pa^2 G_k}{\pa z \pa \ov \zeta}(z_k,\eta_k)
			\ov{T(\eta_k)}\abs{\zeta_k-\eta_k} +
				{O}\big(\abs{\zeta_k - \eta_k}^2\big)}
			{2i\frac{\pa^2 G_k}{\pa z \pa \ov \zeta}(a,\eta_k)
			\ov{T(\eta_k)}\abs{\zeta_k-\eta_k} +
				{O}\big(\abs{\zeta_k - \eta_k}^2\big)} \\
	&= \frac{\frac{\pa^2 G_k}{\pa z \pa \ov \zeta}(z_k,\eta_k) +
				{O}\big(\abs{\zeta_k - \eta_k}\big)}
			{\frac{\pa^2 G_k}{\pa z \pa \ov \zeta}(a,\eta_k) +
				{O}\big(\abs{\zeta_k - \eta_k}\big)}.
\end{align*}
By Proposition 4.1, the above term converges to
\begin{align*}
\frac{\frac{\pa^2 G}{\pa z \pa \ov \zeta}(z_0,\zeta_0)}
{\frac{\pa^2 G}{\pa z \pa \ov \zeta}(z_0,\zeta_0)} = \frac{K(z_0,\zeta_0)}{K(a,\zeta_0)},
\end{align*}
as $k \rightarrow \infty$. The function $F_k(z_k,\zeta_k) = 0$ precisely when $\pa G_k(z_k,\zeta_k) = 0$.
Thus if $(z_k,\zeta_k)$ converges to $(z_0,\zeta_0)$ and $(z_k,\zeta_k)$ are
critical points of $G_k$, then $K(z_0,\zeta_0) = 0$. This proves one part of
the result.

\medskip

For the converse, let $\zeta_k \in D_k$ converge $\zeta_0 \in \pa D$. Fix a disc
$B = B(z_0,r)$ that is compactly contained in $D$ (and hence $D_k$ for large $k$), such that
$F(z, \zeta_0) = K(z,\zeta_0)/K(a,\zeta_0)$ has an isolated zero in $\ov B$.
This is possible since $F(z,\zeta_0)$ is holomorphic. Now observe that $F_k(z,\zeta_k)$ is uniformly bounded on $B$. This is
because $\big\vert F_k(z,\zeta_k)\big\vert_{\pa B} < M_k$ and if $\sup_k M_k$ is not
bounded, then there exists an increasing subsequence $\{k_m\}$ and
$w_{k_m} \in \pa B$ such that
$\big\vert F_{k_m}(w_{k_m},\zeta_{k_m})\big\vert \rightarrow \infty$ as
$m \rightarrow \infty$. But by Proposition 4.1,
$F_{k_m}(w_{k_m},\zeta_{k_m}) \rightarrow F(w,\zeta_0)$ which is a contradiction.
Since $F_k(z,\zeta_k)$ converges pointwise to $F(z,\zeta_0)$ in $\ov B$ and they are a family of uniformly bounded holomorphic functions, the
convergence is uniform on compact sets of $B$. Hence by Hurwitz's theorem,
there exists a subsequence $\{k_m\}$ such that $F_{k_m}(z_{k_m}, \zeta_{k_m}) = 0$ which shows that $(z_{k_m}, \zeta_{k_m})$ are critical points for $G_k$, and this completes the proof of the theorem.


\begin{bibdiv}
\begin{biblist}

\bib{Ab}{article}{
   author={Aboudi, Nabil},
   title={Geodesics for the capacity metric in doubly connected domains},
   journal={Complex Var. Theory Appl.},
   volume={50},
   date={2005},
   number={1},
   pages={7--22},
   issn={0278-1077},
   review={\MR{2114349}},
   doi={10.1080/02781070412331327892},
}
\bib{Bl}{article}{
   author={B{\l}ocki, Zbigniew},
   title={Suita conjecture and the Ohsawa-Takegoshi extension theorem},
   journal={Invent. Math.},
   volume={193},
   date={2013},
   number={1},
   pages={149--158},
   issn={0020-9910},
   review={\MR{3069114}},
   doi={10.1007/s00222-012-0423-2},
}

\bib{D}{article}{
   author={Donnelly, Harold},
   title={$L_2$ cohomology of pseudoconvex domains with complete K\"ahler
   metric},
   journal={Michigan Math. J.},
   volume={41},
   date={1994},
   number={3},
   pages={433--442},
   issn={0026-2285},
   review={\MR{1297700}},
   doi={10.1307/mmj/1029005071},
}
\bib{DF}{article}{
   author={Donnelly, Harold},
   author={Fefferman, Charles},
   title={$L^{2}$-cohomology and index theorem for the Bergman metric},
   journal={Ann. of Math. (2)},
   volume={118},
   date={1983},
   number={3},
   pages={593--618},
   issn={0003-486X},
   review={\MR{727705}},
   doi={10.2307/2006983},
}

\bib{Fe}{article}{
   author={Fefferman, Charles},
   title={The Bergman kernel and biholomorphic mappings of pseudoconvex
   domains},
   journal={Invent. Math.},
   volume={26},
   date={1974},
   pages={1--65},
   issn={0020-9910},
   review={\MR{0350069 (50 \#2562)}},
}

\bib{GS}{article}{
   author={Gustafsson, Bj{\"o}rn},
   author={Sebbar, Ahmed},
   title={Critical points of Green's function and geometric function theory},
   journal={Indiana Univ. Math. J.},
   volume={61},
   date={2012},
   number={3},
   pages={939--1017},
   issn={0022-2518},
   review={\MR{3071691}},
   doi={10.1512/iumj.2012.61.4621},
}

\bib{He}{article}{
   author={Herbort, Gregor},
   title={On the geodesics of the Bergman metric},
   journal={Math. Ann.},
   volume={264},
   date={1983},
   number={1},
   pages={39--51},
   issn={0025-5831},
   review={\MR{709860}},
   doi={10.1007/BF01458049},
}

\bib{H}{book}{
   author={Herv{\'e}, Michel},
   title={Analytic and plurisubharmonic functions in finite and infinite
   dimensional spaces},
   series={Lecture Notes in Mathematics, Vol. 198},
   note={Course given at the University of Maryland, College Park, Md.,
   Spring 1970},
   publisher={Springer-Verlag, Berlin-New York},
   date={1971},
   pages={vi+90},
   review={\MR{0466602}},
}

\bib{L}{article}{
   author={Lelong, Pierre},
   title={Fonctions plurisousharmoniques et fonctions analytiques de
   variables r\'eelles},
   language={French},
   journal={Ann. Inst. Fourier (Grenoble)},
   volume={11},
   date={1961},
   pages={515--562},
   issn={0373-0956},
   review={\MR{0142789}},
}

\bib{LY}{article}{
   author={Levenberg, Norman},
   author={Yamaguchi, Hiroshi},
   title={The metric induced by the Robin function},
   journal={Mem. Amer. Math. Soc.},
   volume={92},
   date={1991},
   number={448},
   pages={viii+156},
   issn={0065-9266},
   review={\MR{1061928 (91m:32017)}},
}

\bib{MY}{article}{
   author={Maitani, Fumio},
   author={Yamaguchi, Hiroshi},
   title={Variation of Bergman metrics on Riemann surfaces},
   journal={Math. Ann.},
   volume={330},
   date={2004},
   number={3},
   pages={477--489},
   issn={0025-5831},
   review={\MR{2099190}},
   doi={10.1007/s00208-004-0556-8},
}

\bib{Mc}{article}{
   author={McNeal, Jeffery D.},
   title={$L^2$ harmonic forms on some complete K\"ahler manifolds},
   journal={Math. Ann.},
   volume={323},
   date={2002},
   number={2},
   pages={319--349},
   issn={0025-5831},
   review={\MR{1913045}},
   doi={10.1007/s002080100305},
}

\bib{Mi2}{article}{
   author={Minda, David},
   title={Inequalities for the hyperbolic metric and applications to
   geometric function theory},
   conference={
      title={Complex analysis, I},
      address={College Park, Md.},
      date={1985--86},
   },
   book={
      series={Lecture Notes in Math.},
      volume={1275},
      publisher={Springer, Berlin},
   },
   date={1987},
   pages={235--252},
   review={\MR{922304}},
   doi={10.1007/BFb0078356},
}

\bib{Mi1}{article}{
   author={Minda, David},
   title={The capacity metric on Riemann surfaces},
   journal={Ann. Acad. Sci. Fenn. Ser. A I Math.},
   volume={12},
   date={1987},
   number={1},
   pages={25--32},
   issn={0066-1953},
   review={\MR{877576}},
   doi={10.5186/aasfm.1987.1226},
}

\bib{Oh}{article}{
   author={Ohsawa, Takeo},
   title={On the infinite dimensionality of the middle $L^2$ cohomology
   of complex domains},
   journal={Publ. Res. Inst. Math. Sci.},
   volume={25},
   date={1989},
   number={3},
   pages={499--502},
   issn={0034-5318},
   review={\MR{1018512}},
   doi={10.2977/prims/1195173354},
}

\bib{RT}{article}{
   author={Rodr{\'{\i}}guez, J. M.},
   author={Tour{\'{\i}}s, E.},
   title={Gromov hyperbolicity through decomposition of metric spaces},
   journal={Acta Math. Hungar.},
   volume={103},
   date={2004},
   number={1-2},
   pages={107--138},
   issn={0236-5294},
   review={\MR{2047877}},
   doi={10.1023/B:AMHU.0000028240.16521.9d},
}

\bib{So}{article}{
   author={Solynin, Alexander Yu.},
   title={A note on equilibrium points of Green's function},
   journal={Proc. Amer. Math. Soc.},
   volume={136},
   date={2008},
   number={3},
   pages={1019--1021 (electronic)},
   issn={0002-9939},
   review={\MR{2361876}},
   doi={10.1090/S0002-9939-07-09156-3},
}



\bib{Su}{article}{
   author={Suita, Nobuyuki},
   title={Capacities and kernels on Riemann surfaces},
   journal={Arch. Rational Mech. Anal.},
   volume={46},
   date={1972},
   pages={212--217},
   issn={0003-9527},
   review={\MR{0367181}},
}

\bib{Ya}{article}{
   author={Yamaguchi, Hiroshi},
   title={Sur le mouvement des constantes de Robin},
   language={French},
   journal={J. Math. Kyoto Univ.},
   volume={15},
   date={1975},
   pages={53--71},
   issn={0023-608X},
   review={\MR{0382695}},
}

\bib{Yau}{article}{
   author={Yau, Shing Tung},
   title={Some function-theoretic properties of complete Riemannian manifold
   and their applications to geometry},
   journal={Indiana Univ. Math. J.},
   volume={25},
   date={1976},
   number={7},
   pages={659--670},
   issn={0022-2518},
   review={\MR{0417452}},
}

\end{biblist}
\end{bibdiv}
\end{document}